\documentclass[notitlepage,11pt,psamsfonts]{amsart}

\usepackage{amsmath,amssymb,amsthm}
\usepackage[latin1]{inputenc}
\usepackage{verbatim}

\usepackage{bm}			

\newtheorem{thm}{Theorem}
\newtheorem{lem}{Lemma}
\newtheorem{cor}{Corollary}
\newtheorem{prop}{Proposition}
\newtheorem{rem}{Remark}
\newtheorem{defi}{Definition}

\newcommand{\eps}{\varepsilon}
\newcommand{\R}{\mathbb{R}}

\newcommand{\N}{\mathbb{N}}
\newcommand{\Z}{\mathbb{Z}}

\renewcommand{\div}{{\rm div}\,}

\newcommand{\Id}{{\rm Id}\,}

\newcommand{\Sum}{\displaystyle \sum}

\newcommand{\Int}{\displaystyle \int}

\newcommand{\mbb}{\mathbb}                                      
\newcommand{\mc}{\mathcal}					
\newcommand{\veps}{\varepsilon}					
\newcommand{\wtilde}{\widetilde}				

\def\ov{\overline}
\def\d{\partial}

\def\dj{\Delta_j}
\def\tilde{\widetilde}
\def\hat{\widehat}
\def\div{{\rm div}\,}
\def\s{\sigma}
\def\cA{{\mathcal A}}

\def\cC{{\mathcal C}}
\def\cF{{\mathcal F}}

\def\cP{{\mathcal P}}
\def\cQ{{\mathcal Q}}
\def\cS{{\mathcal S}}
\def\da{\delta\!a}

\def\du{\delta\!u}

\begin{document}
\title[Density-dependent incompressible Euler equations]
{THE WELL-POSEDNESS ISSUE FOR THE  DENSITY-DEPENDENT
EULER EQUATIONS IN ENDPOINT BESOV SPACES}

\author[R. Danchin]{Rapha\"el Danchin}
\address[R. Danchin]
{Universit\'e Paris-Est, LAMA, UMR 8050,
 61 avenue du G\'en\'eral de Gaulle,
94010 Cr\'eteil Cedex, France.}
\email{danchin@univ-paris12.fr}

\author[F. Fanelli]{Francesco Fanelli}
\address[F. Fanelli]
{SISSA,
via Bonomea 265,
34136 Trieste, Italy.}
\email{francesco.fanelli@sissa.it}

\date\today
\begin{abstract}
This work  is the continuation of the recent paper \cite{D2}
devoted to the density-dependent incompressible Euler equations. 
Here we concentrate on the well-posedness issue in Besov spaces
of type $B^s_{\infty,r}$ embedded in the set of Lipschitz continuous functions, 
a functional framework which contains the particular case of H\"older spaces
and of the endpoint Besov space $B^1_{\infty,1}.$ 
For such data and under the nonvacuum assumption, we establish 
the local well-posedness and a continuation criterion in the spirit of
that of Beale, Kato and Majda in \cite{BKM}.
 
In the last part of the paper, we give lower bounds for the lifespan of a solution. 
In dimension two, we point out that the lifespan tends to infinity 
when the initial density tends to be a constant.
This is, to our knowledge, the first result of this kind for the
density-dependent incompressible Euler equations. 
  \end{abstract}

\maketitle

\section{Introduction and main results}

This work is the continuation of a recent paper by the first author (see \cite{D2})
devoted to the 
\emph{density-dependent incompressible Euler equations}:
\begin{equation}\label{eq:ddeuler}
\left\{\begin{array}{l}
\d_t\rho+u\cdot\nabla \rho=0,\\[1ex]
\rho(\d_tu+u\cdot\nabla u)+\nabla\Pi=\rho f,\\[1ex]
\div u=0.
\end{array}\right.
\end{equation}

Recall that the above equations describe the evolution of  the density $\rho=\rho(t,x)\in\R_+$
and of the velocity field  $u=u(t,x)\in\R^N$ of 
a nonhomogeneous inviscid incompressible fluid. 
The time dependent vector field  $f$ stands for a given body force and the gradient of the  pressure $\nabla\Pi$
is the Lagrangian multiplier associated to the divergence free constraint over the velocity. 
We assume that the space variable $x$  belongs to the whole $\R^N$ with $N\geq2$. 
\smallbreak
There is an important  literature devoted to the  standard incompressible Euler equations, 
that is to the case where the initial density is a positive constant, an assumption which is preserved during the evolution. 
In contrast, not so many works have been devoted to the study of \eqref{eq:ddeuler}
in the nonconstant density case.  In the situation  
where the  equations are considered in  a suitably smooth bounded 
 domain of $\R^2$ or $\R^3,$ the local
well-posedness issue has been investigated by H. Beir\~ao da Veiga and A. Valli
 in \cite{BV1,BV2,BV3} for data 
 with high enough  H\"older regularity.
 In  \cite{D1}, we have proved well-posedness in  $H^s$ with $s>1+N/2$ and have studied the inviscid limit
 in this framework.  Data in the limit Besov space $B^{\frac N2+1}_{2,1}$
 were  also considered.  
  
 As for the standard incompressible Euler equations, 
 any functional space   embedded in 
the set $C^{0,1}$ of bounded globally Lipschitz functions
is a candidate  for the study of the well-posedness issue. 
This stems from  the fact that System \eqref{eq:ddeuler}  is a  coupling between
  transport equations. 
Hence preserving the initial regularity requires the velocity field to be at least locally Lipschitz with respect to the space variable.
As a matter of fact, 
the classical Euler equations have been shown to be well posed in 
any Besov space $B^s_{p,r}$  embedded in $C^{0,1}$ 
(see \cite{BCD,Ch,PP,Z} and the references therein),
a property which holds if and only if  
 $(s,p,r)\in\mbb{R}\times[1,+\infty]^2$ satisfies
$$
s>1+\frac{N}{p}\qquad\qquad\mbox{or}\qquad\qquad s=1+\frac{N}{p}\;\;\mbox{and}\;\;r=1\,. \leqno(C)
$$

In \cite{D2}, we extended the results of the homogeneous case to \eqref{eq:ddeuler}
(see also \cite{FXZ} for a similar study in the periodic framework).
Under condition $(C)$ \emph{with $1<p<\infty$} we established the local well-posedness for any data $(\rho_0,u_0)$ in $B^s_{p,r}$
 such that  $\rho_0$ is bounded away from zero.  
 However, we have been unable to treat the limit case $p=\infty$
 \emph{unless the initial density is a small perturbation of a constant density state},
  a technical artifact due to the method
 we used to handle the pressure term. 
 
 In fact,  in contrast to the classical Euler equations, computing the gradient of the pressure
involves 
an  elliptic equation \emph{with nonconstant coefficients}, namely
 \begin{equation}\label{eq:ell}
 \div\bigl(a\nabla\Pi\bigr)=\div F\quad\hbox{with }\ F:=\div(f-u\cdot\nabla u)\ \hbox{ and }\
 a:=1/\rho.
 \end{equation}
 
 Getting  appropriate a priori estimates
 \emph{given that we expect the function $\rho$ to have exactly the same regularity
 as $\nabla\Pi$} is the main difficulty. 
In the $L^2$ framework and, more generally, in the Sobolev 
framework $H^s,$ this may be achieved by means of a classical energy method.
This is also quite straightforward in the $B^s_{p,r}$ framework  if 
 $a$  is  a small perturbation of  some positive constant function $\ov a,$ for
 the above equation may be rewritten 
 $$
 \ov a\Delta\Pi= 
 \div F+\div\bigl((\ov a-a)\nabla\Pi\bigr).
 $$
 
 Now, if $a-\ov a$ is small enough then one may take advantage of regularity results for the Laplace 
 operator in order to ``absorb'' the last term. 
 
 If $1<p<\infty$ and $a$ is bounded away from zero then
 it turns out that combining energy arguments similar to those of the $H^s$ case
 and a harmonic analysis lemma allows to handle the elliptic equation \eqref{eq:ell}.
 This is  the approach that we used in \cite{D2}. 
 However it fails for the limit cases $p=1$ and $p=\infty.$
  
 \medbreak
 In the present work, we propose another  method for proving  a priori estimates for 
 \eqref{eq:ell}.
 In addition to being  simpler, this will enable us to treat
 all the cases $p\in[1,\infty]$ indistinctly whenever the density is bounded away 
 from zero. 
  Our approach relies on the fact that the pressure $\Pi$ satisfies (here we take $f\equiv0$
  to simplify)
  \begin{equation}\label{eq:pressure}
  \Delta\Pi=-\rho\,\div(u\cdot\nabla u)+\nabla\log\rho\cdot\nabla\Pi.
 \end{equation}
 Obviously, the last term is of lower order. 
 In addition, the classical $L^2$ theory ensures
 that 
 $$
 a_*\|\nabla\Pi\|_{L^2}\leq \|u\cdot\nabla u\|_{L^2}\quad\hbox{with }\ 
 a_*:=\inf_{x\in\R^N} a(x).
 $$
 Therefore interpolating between the high regularity estimates for the Laplace operator
 and the $L^2$ estimate allows to absorb the last term in 
 the right-hand side of \eqref{eq:pressure}.
 \smallbreak
 In the rest of the paper, we focus on the case $p=\infty$
 as it is the only definitely new one and 
 as it covers both  H\"older spaces with 
 exponent greater than $1$
 and the limit space $B^1_{\infty,1}$ which is the largest one in 
 which one may expect to get well-posedness.

 \medbreak  
 
 Before going further into the description of our results, let us introduce a few notation.
 \begin{itemize}
\item Throughout the paper, $C$ stands for
a harmless ``constant'' the meaning of which depends on the context.
\item
If $a=(a^1,a^2)$ and $b=(b^1,b^2)$ then we denote
$
a\wedge b:=a^1b^2-a^2b^1.
$
\item 
The vorticity $\Omega$ associated to a vector field $u$ over $\R^N$ is the matrix valued
function with entries
$$
\Omega_{ij}:=\d_j u^i-\d_iu^j.
$$
If $N=2$ then the vorticity may be identified with the scalar function
$\omega:=\d_1u^2-\d_2u^1$ and if $N=3,$ with the vector field $\nabla\times u.$
\item For all  Banach space $X$ and interval $I$ of $\R,$  
we denote by $\cC(I;X)$ (resp. $\cC_b(I;X)$)
 the set of continuous  (resp. continuous bounded) functions on $I$ with
values in $X.$
If $X$ has predual $X^*$ then 
we denote by $\cC_w(I;X)$ the set of bounded measurable functions 
$f:I\rightarrow X$ such that for any $\phi\in X^*,$ the 
function $t\mapsto\langle f(t),\phi\rangle_{X\times X^*}$
is continuous over~$I.$
\item For $p\in[1,\infty]$, the notation $L^p(I;X)$ 
stands for the set
 of measurable functions on  $I$ with values in $X$ such that
$t\mapsto \|f(t)\|_X$ belongs to $L^p(I)$.
In the case $I=[0,T]$ we alternately use the notation 
$L_T^p(X).$

\item We denote by $L^p_{loc}(I)$ the set of those functions defined on $I$ and valued in $X$
which, restricted to any compact subset $J$ of $I,$ are in $L^p(J).$

\item Finally, for any real valued function $a$ over $\R^N,$ we 
denote
$$
a_*:=\inf_{x\in\R^N}a(x)\quad\hbox{and}\quad
a^*:=\sup_{x\in\R^N}a(x).
$$
\end{itemize}
\medbreak

Let us now state our main well-posedness result
in the  case of a finite energy initial velocity field.
\begin{thm} \label{th:L^2}
Let $r$ be in $[1,\infty]$ and $s\in\R$ satisfy $s>1$ if $r\not=1$
and $s\geq1$ if $r=1.$ 
Let $\rho_0$ be a positive function in $B^{s}_{\infty,r}$ bounded away from $0,$
 and $u_0$ be a divergence-free vector field
with coefficients in $B^{s}_{\infty,r}\cap L^2$.
Finally, suppose that the external force $f$ has coefficients in $L^1([-T_0,T_0];B^{s}_{\infty,r})\cap\cC([-T_0,T_0];L^2)$ for some positive
time $T_0$.

Then there exists a time $T\in]0,T_0]$ such that System \eqref{eq:ddeuler} with initial data $(\rho_0,u_0)$ has a unique solution
$(\rho,u,\nabla\Pi)$ on $[-T,T]\times\mbb{R}^N$, with:
\begin{itemize}
\item $\rho$ in $\cC([-T,T];B^s_{\infty,r})$ and bounded away from $0,$
\item $u$ in $\cC([-T,T];B^s_{\infty,r})\cap \cC^1([-T,T];L^2)$ and
\item $\nabla\Pi$ in $L^1([-T,T];B^s_{\infty,r})\cap \cC([-T,T];L^2)$.
\end{itemize}
If $r=\infty$ then one has only weak continuity in time with values in the Besov space $B^s_{\infty,\infty}$.
\end{thm}

\medbreak
In the above functional framework, one may state a continuation criterion
for the solution to \eqref{eq:ddeuler} similar to that of Theorem 2 of  \cite{D2}:

\begin{thm}\label{th:cont}
 Let $(\rho,u,\nabla\Pi)$ be a solution to System \eqref{eq:ddeuler} on $[0,T^*[\times\mbb{R}^N,$ 
 with the properties described in Theorem \ref{th:L^2} for all $T<T^*$; suppose also that we have 
\begin{equation} \label{eq:cont_cond}
 \int^{T^*}_0\left(\|\nabla u\|_{L^\infty}\,+\,\|\nabla\Pi\|_{B^{s-1}_{\infty,r}}\right)dt\,<\,\infty\,.
\end{equation}

If $T^*$ is finite then $(\rho,u,\nabla\Pi)$ can be continued beyond $T^*$ into a solution of \eqref{eq:ddeuler} with the same regularity.
Moreover, if $s>1$ then one may replace in \eqref{eq:cont_cond} the term $\|\nabla u\|_{L^\infty}$ with $\|\Omega\|_{L^\infty}.$

A similar result holds for negative times.
\end{thm}

{}From this result, as our assumption on $(r,s)$  implies that $B^{s-1}_{\infty,r}\hookrightarrow L^\infty$, keeping in mind that
$B^1_{\infty,1}$ is the largest Besov space included in $C^{0,1}$, we immediately get the following:

\begin{cor} \label{C:lifespan}
 The lifespan of a solution in $B^s_{\infty,r}$ with $s>1$ is the same as the lifespan in $B^1_{\infty,1}$.
\end{cor}

\smallbreak
As pointed out in \cite{D2},  hypothesis $u_0\in L^2$ is somewhat restrictive  in dimension $N=2$
as  if, say,  the initial vorticity  $\omega_0$ is in $L^1$
then  it implies that $\omega_0$ has average $0$ over $\mbb{R}^2$.
In particular, assuming that $u_0\in L^2(\R^2)$ precludes our considering
general data with initially compactly supported nonnegative vorticity (e.g.  vortex patches
as in \cite{Ch}, Chapter 5).

The following statement aims at considering 
initial data \emph{with infinite energy}. For simplicity,
we suppose the external force to be $0.$

\begin{thm} \label{th:W^1,4}
 Let $(s,r)$ be as in Theorem \ref{th:L^2}.
Let $\rho_0\in B^{s}_{\infty,r}$ be bounded away from $0,$
 and $u_0\in B^{s}_{\infty,r}\cap W^{1,4}$.

Then there exist a positive time $T$ and a unique solution $(\rho,u,\nabla\Pi)$ on $[-T,T]\times\mbb{R}^N$ of System \eqref{eq:ddeuler}
with external force $f\equiv0$, satisfying the following properties:
\begin{itemize}
\item $\rho\,\in\,\mc{C}([-T,T];B^s_{\infty,r})$ bounded away from $0,$
\item $u\,\in\,\mc{C}([-T,T];B^s_{\infty,r}\cap W^{1,4})$ and $\d_tu\in\cC([-T,T];L^2),$
\item $\nabla\Pi\,\in\,L^1([-T,T];B^s_{\infty,r})\cap\mc{C}([-T,T];L^2)$.
\end{itemize}
As above, the continuity in time with values in $B^s_{\infty,r}$ is only weak  if $r=\infty$.
\end{thm}
\begin{rem} 
Under the above hypothesis, a continuation criterion in the spirit
of Theorem \ref{th:cont} may be proved. The details are left to the reader. 

Let us also point out that  there is some freedom over the $W^{1,4}$ assumption (see Remark
\ref{r:W^1,4} below). 
\end{rem}
 
 On the one hand,  the existence results that we stated so 
 far  are \emph{local in time} even  in the two-dimensional case.
 On the other hand, it is well known that the classical two-dimensional incompressible Euler equations
 are globally  well-posed, a result that goes back to the pioneering work by V. Wolibner
 in \cite{W}  (see also  \cite{Yu,hmidiK,V} for global results in the case of less regular data).
 In the homogeneous case, the global 
 existence stems from the fact that  the vorticity $\omega$ is transported by the flow associated to the solution: we have
$$
\partial_t\omega\,+\,u\cdot\nabla\omega\,=\,0\,.
$$
In the nonhomogeneous context  this relation is not true any longer: 
we have instead
\begin{equation}\label{eq:vorticity}
\partial_t\omega\,+\,u\cdot\nabla\omega\,+\,\nabla\left(\frac{1}{\rho}\right)\wedge\nabla\Pi\,=\,0\,.
\end{equation}

If the classical homogeneous case has been deeply studied, to our knowledge there is no literature about the time of existence of solutions for the density-dependent
incompressible Euler system. In the last section of this paper, we 
establish lower bounds for the lifespan
of a solution of \eqref{eq:ddeuler}.

Roughly, we show that  in  any space dimension,  if the initial velocity is 
of order $\veps$ ($\veps$ small enough),
\emph{without any restriction on the density of the fluid} then the lifespan is 
at least of order $\veps^{-1}$ (see the exact statement in Theorem \ref{th:ND}). 

Next, taking advantage of Equality \eqref{eq:vorticity} and of an
estimate for the transport equation that has been established recently by M. Vishik in \cite{V}
(and generalized by T. Hmidi and S. Keraani in \cite{HK}), we show  that 
the lifespan of the solution 
tends to infinity if $\rho_0-1$ goes to $0.$ 
More precisely, Theorem \ref{th:2D} states that if 
$$
\|\rho_0-1\|_{B^1_{\infty,1}}=\eta\quad\hbox{and}\quad
\|\omega_0\|_{B^0_{\infty,1}}+\|u_0\|_{L^2}=\eps
$$
with $\eta$ small enough,
then  the lifespan is at least of order
$\eps^{-1}\log(\log\eta^{-1}).$

\medbreak

The paper is organized as follows. In the next  section, we introduce the tools
needed for proving our results: the Littlewood-Paley decomposition, the definition of the nonhomogeneous Besov spaces $B^s_{p,r}$ and the
paradifferential calculus, and finally some classical results 
about transport equations in $B^s_{p,r}$ and  elliptic equations. Sections \ref{s:L^2} and \ref{s:W^1,4}
  are devoted to the proof of our local existence statements
  first in the finite energy case and next if the initial velocity is in $W^{1,4}.$ 
  Finally, in the last section we  state and prove  results about the
lifespan of a solution of our system, focusing on the particular  case of space dimension $N=2$.


\section{Tools}\label{s:tools}

Our results mostly rely on  Fourier analysis methods
based on  a nonhomogeneous dyadic partition of unity 
with respect to the Fourier variable, the so-called Littlewood-Paley decomposition. 
Unless otherwise specified, all the results which are presented
in this section are proved in \cite{BCD}.

In order to define a  Littlewood-Paley decomposition, 
fix a smooth radial function
$\chi$ supported in (say) the ball $B(0,\frac43),$ 
equals to $1$ in a neighborhood of $B(0,\frac34)$
and such that $r\mapsto\chi(r\,e_r)$ is nonincreasing
over $\R_+,$ and set
$\varphi(\xi)=\chi(\frac\xi2)-\chi(\xi).$
\smallbreak
The {\it dyadic blocks} $(\Delta_j)_{j\in\Z}$
 are defined by\footnote{Throughout we agree  that  $f(D)$ stands for 
the pseudo-differential operator $u\mapsto\cF^{-1}(f\cF u).$} 
$$
\dj:=0\ \hbox{ if }\ j\leq-2,\quad\Delta_{-1}:=\chi(D)\quad\hbox{and}\quad
\Delta_j:=\varphi(2^{-j}D)\ \text{ if }\  j\geq0.
$$
We  also introduce the following low frequency cut-off:
$$
S_ju:=\chi(2^{-j}D)=\sum_{j'\leq j-1}\Delta_{j'}\quad\text{for}\quad j\geq0.
$$
The following classical properties will be used freely throughout in the paper:
\begin{itemize}
\item for any $u\in\cS',$ the equality $u=\sum_{j}\dj u$ holds true in $\cS'$;
\item for all $u$ and $v$ in $\cS',$
the sequence
$(S_{j-1}u\,\dj v)_{j\in\N}$ is spectrally
supported in dyadic annuli.
\end{itemize}
One can now define what a Besov space $B^s_{p,r}$ is:
\begin{defi}
\label{def:besov}
  Let  $u$ be a tempered distribution, $s$ a real number, and 
$1\leq p,r\leq\infty.$ We set
$$
\|u\|_{B^s_{p,r}}:=\bigg(\sum_{j} 2^{rjs}
\|\Delta_j  u\|^r_{L^p}\bigg)^{\frac{1}{r}}\ \text{ if }\ r<\infty
\quad\text{and}\quad
\|u\|_{B^s_{p,\infty}}:=\sup_{j}\left( 2^{js}
\|\Delta_j  u\|_{L^p}\right).
$$
We then define the space $B^s_{p,r}$ as  the
subset of  distributions $u\in {\cS}'$ such  that
$\|u\|_{B^s_{p,r}}$ is finite.
\end{defi}
  
 {}From the above definition, it is easy to show that for all $s\in\R,$ the Besov space $B^s_{2,2}$ coincides
  with the nonhomogeneous Sobolev space $H^s.$
  Let us also point out that for any $k\in\N$ and $p\in[1,\infty],$
   we have the following chain of continuous embedding:
 $$
 B^k_{p,1}\hookrightarrow W^{k,p}\hookrightarrow B^k_{p,\infty}.
 $$
  where  $W^{k,p}$ denotes the set of $L^p$ functions
 with derivatives up to order $k$ in $L^p.$
\medbreak

The Besov spaces have many nice  properties which will be recalled throughout the paper
whenever they are needed.
For the time being, let us just recall that if Condition $(C)$ holds true
then $B^s_{p,r}$ is an algebra continuously embedded in the set $C^{0,1}$
of bounded Lipschitz functions (see e.g. \cite{BCD}, Chap. 2), and that the gradient operator
maps $B^s_{p,r}$ in $B^{s-1}_{p,r}.$

\medbreak

The following result will be also needed:
\begin{prop}\label{p:CZ}
Let $F:\R^N\rightarrow\R$
 be a smooth homogeneous function of degree $m$ away from a neighborhood of the origin.
 Then for all $(p,r)\in[1,\infty]^2$ and $s\in\R,$  Operator $F(D)$ maps $B^s_{p,r}$
 in $B^{s-m}_{p,r}.$ 
 \end{prop}
 
   \begin{rem}\label{r:CZ}
  Let $\cP$ be   the Leray projector  over divergence free vector fields
  and $\cQ:={\rm Id}-\cP.$
  Recall that   in Fourier variables, we have for all vector field $u$
   $$
 \hat{\cQ u}(\xi)=-\frac{\xi}{|\xi|^2}\,\xi\cdot\hat u(\xi).
 $$
  Therefore, both $(\Id-\Delta_{-1})\cP$ and 
  $(\Id-\Delta_{-1})\cQ$
     satisfy the assumptions of the above proposition with $m=0$ 
   hence are self-map on $B^s_{p,r}$ \emph{for any $s\in\R$ and $1\leq p,r\leq\infty.$} 
  \end{rem}
  The following lemma (referred in what follows as \emph{Bernstein's inequalities})
  describes the way derivatives act on spectrally localized functions.
  \begin{lem}
\label{lpfond}
{\sl
Let  $0<r<R.$   A
constant~$C$ exists so that, for any nonnegative integer~$k$, any couple~$(p,q)$ 
in~$[1,\infty]^2$ with  $q\geq p\geq 1$ 
and any function $u$ of~$L^p$,  we  have for all $\lambda>0,$
$$
\displaylines{
{\rm Supp}\, \widehat u \subset   B(0,\lambda R)
\Longrightarrow
\|\nabla^k u\|_{L^q} \leq
 C^{k+1}\lambda^{k+N(\frac{1}{p}-\frac{1}{q})}\|u\|_{L^p};\cr
{\rm Supp}\, \widehat u \subset \{\xi\in\R^N\,/\, r\lambda\leq|\xi|\leq R\lambda\}
\Longrightarrow C^{-k-1}\lambda^k\|u\|_{L^p}
\leq
\|\nabla^k u\|_{L^p}
\leq
C^{k+1}  \lambda^k\|u\|_{L^p}.
}$$
}
\end{lem}   
  The first Bernstein inequality entails the following embedding result:
  \begin{cor}\label{c:embed}
  The space $B^{s_1}_{p_1,r_1}$ is continuously embedded in the space $B^{s_2}_{p_2,r_2}$ whenever
  $1\leq p_1\leq p_2\leq\infty$ and
  $$
  s_2< s_1-N/p_1+N/p_2\quad\hbox{or}\quad
  s_2=s_1-N/p_1+N/p_2\ \hbox{ and }\ 1\leq r_1\leq r_2\leq\infty. 
  $$
  \end{cor}
 Let us now introduce the paraproduct operator and recall a few nonlinear estimates in Besov spaces. 
Constructing the paraproduct operator relies on the observation that, 
formally, any product  of two tempered distributions $u$ and $v,$ may be decomposed
into 
\begin{equation}\label{eq:bony}
uv=T_uv+T_vu+R(u,v)
\end{equation}
with 
$$
T_uv:=\sum_jS_{j-1}u\dj v,\quad
T_vu:=\sum_j S_{j-1}v\dj u\ \hbox{ and }\ 
R(u,v):=\sum_j\sum_{|j'-j|\leq1}\dj u\,\Delta_{j'}v.
$$
The above operator $T$ is called ``paraproduct'' whereas
$R$ is called ``remainder''.
\smallbreak
The paraproduct and remainder operators have many nice continuity properties. 
The following ones will be of constant use in this paper (see the proof in e.g. \cite{BCD}, Chap. 2):
\begin{prop}\label{p:op}
For any $(s,p,r)\in\R\times[1,\infty]^2$ and $t<0,$ the paraproduct operator 
$T$ maps $L^\infty\times B^s_{p,r}$ in $B^s_{p,r},$
and  $B^t_{\infty,\infty}\times B^s_{p,r}$ in $B^{s+t}_{p,r}.$
Moreover, the following estimates hold:
$$
\|T_uv\|_{B^s_{p,r}}\leq C\|u\|_{L^\infty}\|\nabla v\|_{B^{s-1}_{p,r}}\quad\hbox{and}\quad
\|T_uv\|_{B^{s+t}_{p,r}}\leq C\|u\|_{B^t_{\infty,\infty}}\|\nabla v\|_{B^{s-1}_{p,r}}.
$$
For any $(s_1,p_1,r_1)$ and $(s_2,p_2,r_2)$ in $\R\times[1,\infty]^2$ such that 
$s_1+s_2>0,$ $1/p:=1/p_1+1/p_2\leq1$ and $1/r:=1/r_1+1/r_2\leq1$
the remainder operator $R$ maps 
$B^{s_1}_{p_1,r_1}\times B^{s_2}_{p_2,r_2}$ in $B^{s_1+s_2}_{p,r}.$
\end{prop}
Combining the above proposition with Bony's decomposition \eqref{eq:bony}, 
we easily get the following ``tame estimate'':
\begin{cor}\label{c:op}
Let $a$ be a bounded function such that $\nabla a\in B^{s-1}_{p,r}$ for some $s>0$
and $(p,r)\in[1,\infty]^2.$  Then for any $b\in B^s_{p,r}\cap L^\infty$ we have $ab\in B^s_{p,r}\cap L^\infty$
and there exists a constant $C$ depending only on $N,$ $p$ and $s$
such that 
$$
\|ab\|_{B^s_{p,r}}\leq C\Bigl(\|a\|_{L^\infty}\|b\|_{B^s_{p,r}}+\|b\|_{L^\infty}\|\nabla a\|_{B^{s-1}_{p,r}}\Bigr).
$$
\end{cor}
The following result pertaining to the composition of functions
in Besov spaces will be needed for estimating the reciprocal  of the density (see the proof 
in \cite{D2}).
\begin{prop}\label{p:comp}
Let $I$ be an open  interval of $\R$  
and $F:I\rightarrow\R,$ a smooth function. 
Then for all compact subset $J\subset I,$ $s>0$ and $(p,r)\in[1,\infty]^2$ there exists a constant $C$
such that for all function $a$ valued in $J$ and with gradient in $B^{s-1}_{p,r},$  we have
$\nabla(F(a))\in B^{s-1}_{p,r}$ and 
$$
\|\nabla(F(a))\|_{B^{s-1}_{p,r}}\leq C\|\nabla a\|_{B^{s-1}_{p,r}}.
$$
\end{prop}

Our results concerning Equations \eqref{eq:ddeuler} rely strongly on a priori estimates
in Besov spaces for the transport equation
$$
\left\{\begin{array}{l}
\d_ta+v\cdot\nabla a=f,\\[1ex]
a_{|t=0}=a_0.\end{array}\right.\leqno(T)
$$
We shall often use the  following result, the proof of which  may be found in e.g. \cite{BCD}.  
\begin{prop}\label{p:transport}  Let $1\leq r\leq\infty$ and 
$\sigma>0$ ($\sigma>-1$ if $\div v=0$). 
Let $a_0\in B^\s_{\infty,r},$
 $f\in L^1([0,T];B^\s_{\infty,r})$ and  $v$
 be a time dependent vector field in $\cC_b([0,T]\times\R^N)$
such that   
$$
\begin{array}{lllll}
\nabla v&\in&  L^1([0,T];
L^\infty)&\hbox{\rm if}&\sigma<1,\\[1.5ex]
\nabla v&\in& L^1([0,T];B^{\sigma -1}_{\infty,r})
&\hbox{\rm if}&\sigma>1,\quad
\hbox{\rm  or }\  \sigma=r=1.
\end{array}
$$
Then Equation $(T)$ has a unique solution  $a$ in 
\begin{itemize}
\item the space $\cC([0,T];B^{\sigma}_{\infty,r})$ if $r<\infty,$
\item the space $\Bigl(\bigcap_{\sigma'<\sigma} \cC([0,T];B^{\sigma'}_{\infty,\infty})\Bigr)
\bigcap  \cC_w([0,T];B^{\sigma}_{\infty,\infty})$ if $r=\infty.$
\end{itemize}
Moreover,  for all $t\in[0,T],$ we have
\begin{equation}\label{sanspertes1}
e^{-CV(t)}\|a(t)\|_{B^\sigma_{\infty,r}}\leq
\|a_0\|_{B^\sigma_{\infty,r}}+\int_0^t
e^{-CV(t')}
\|f(t')\|_{B^\sigma_{\infty,r}}\,dt'
\end{equation} 
$$\displaylines{
\mbox{with}\quad\!\! V'(t):=\left\{
\begin{array}{l}
\!\|\nabla v(t)\|_{L^\infty}\!\!\!\quad\mbox{if}\!\!\!
\quad\sigma<1,\\[1.5ex]
\!\|\nabla v(t)\|_{B^{\sigma-1}_{\infty,r}}\ \mbox{ if }\
\sigma>1,\quad\!\mbox{or }\ 
\sigma=r=1. 
\end{array}\right.\hfill} $$
 If $a=v$ then, for all $\sigma>0$ ($\sigma>-1$ if $\div v=0$),
    Estimate  \eqref{sanspertes1} holds with
$V'(t):=\|\nabla a(t)\|_{L^\infty}.$
\end{prop}

Finally, we shall make an extensive use of energy  estimates  for the
following elliptic equation:
\begin{equation}\label{eq:elliptic}
-\div(a\nabla\Pi)=\div F\quad\hbox{in }\ \R^N
\end{equation}
where $a=a(x)$ is a given suitably smooth bounded function satisfying 
\begin{equation}\label{eq:ellipticity}
a_*:=\inf_{x\in\R^N}a(x)>0.
\end{equation}

We shall use  the following  result
based on Lax-Milgram's theorem (see the proof in e.g. \cite{D2}).  
\begin{lem}\label{l:laxmilgram}
For all vector field $F$ with coefficients in $L^2,$ there exists a tempered distribution $\Pi,$
unique up to  constant functions, 
such that  $\nabla\Pi\in L^2$ and  
Equation $\eqref{eq:elliptic}$ is satisfied. 
In addition, we have 
\begin{equation}\label{eq:el0}
a_* \|\nabla\Pi\|_{L^2}\leq\|F\|_{L^2}.
\end{equation}
\end{lem}


\section{Proof of Theorem \ref{th:L^2}}\label{s:L^2}

Obviously, one may extend the force term for any time so that it is not restrictive
to assume that $T_0=+\infty.$ 
Owing  to time reversibility of System \eqref{eq:ddeuler}, we can restrict ourselves to the problem of evolution  for positive times  only. 
For convenience we will assume $r<\infty$; for treating the case $r=\infty$, it is enough to replace the strong topology by the weak
topology, whenever regularity up to index $s$ is involved.

We will not work on System (\ref{eq:ddeuler}) directly, but rather on
\begin{equation}\label{eq:ddeuler_a}
\left\{\begin{array}{l}
\d_t a\,+\,u\cdot\nabla a\,=\,0\\[1ex]
\d_tu\,+\,u\cdot\nabla u\,+\,a\nabla\Pi\,=\, f\\[1ex]
-\,\div(a\nabla\Pi)\,=\,\div\left(u\cdot\nabla\mc{P}u\,-\,f\right)\,,
\end{array}\right.
\end{equation}
where we have set $a:=1/\rho$. 
\smallbreak

The equivalence between (\ref{eq:ddeuler}) and (\ref{eq:ddeuler_a}) is given in the
following statement (see \cite{D2}).

\begin{lem} \label{l:syst_a}
 Let $u$ be a vector field with coefficients in $\cC^1([0,T]\times\mbb{R}^N)$ and such that $\mc{Q}u\in \cC^1([0,T];L^2)$.
Suppose also that $\nabla\Pi\in\cC([0,T];L^2)$.
Finally, let $\rho$ be a continuous function on $[0,T]\times\mbb{R}^N$ such that 
\begin{equation}\label{eq:ldv}
0<\rho_*\leq\rho\leq\rho^*.
\end{equation} 
Let $a:=1/\rho$.
If $\div\:u(0,\cdot)\equiv0$ in $\mbb{R}^N$ then $(\rho,u,\nabla\Pi)$ is a solution to \eqref{eq:ddeuler} if and only if
$(a,u,\nabla\Pi)$ is a solution to \eqref{eq:ddeuler_a}.
\end{lem}

We now come to the plan of this section.
First of all, we shall  prove a priori estimates for suitably smooth 
solutions of \eqref{eq:ddeuler} or \eqref{eq:ddeuler_a}. 
Even though those estimates are not needed for proving  Theorem \ref{th:L^2}, 
they will be most helpful to get the existence.
As a matter of fact, the construction of solutions which will be proposed in 
the next subsection amounts to solving inductively a sequence of \emph{linear}
equations. The estimates for those approximate solutions turn out
to be the same as those  for the true solutions. 
In the last two subsections, we shall concentrate on the proof of the uniqueness
part of  Theorem \ref{th:L^2} and of the continuation criterion stated in 
Theorem \ref{th:cont} (up to the endpoint case $s=r=1$ which will be
studied in the next section).

\subsection{A priori estimates}

Let $(a,u,\nabla\Pi)$ be a suitably smooth solution of System (\ref{eq:ddeuler_a}) with the required regularity properties.
In this subsection, we show that on a suitably small time interval (the length
of which depends only on the norms of the data), the norm of $(a,u,\nabla\Pi)$
may be bounded in terms of the data. 

Recall that according to  Proposition \ref{p:comp} the quantities  $\|a\|_{B^s_{\infty,r}}$ and $\|\rho\|_{B^s_{\infty,r}}$ are equivalent under hypothesis \eqref{eq:ldv}. 
This fact will be used repeatedly in what follows. 


\subsubsection{Estimates for the density and the velocity field}

Let us assume for a while that $\div u=0.$ 
Then $(\rho,u,\nabla\Pi)$ satisfies System (\ref{eq:ddeuler})
and the following energy equality  holds true:
\begin{equation} \label{eq:cons_en}
 \|\sqrt{\rho(t)}\,u(t)\|^2_{L^2}\,=\,\|\sqrt{\rho_0}\,u_0\|^2_{L^2}\,+\,2\int^t_0\left(\int_{\mbb{R}^N}\rho\,f\cdot u\,dx\right)d\tau\,.
\end{equation}

Moreover, from the equation satisfied by the density, we have that $\rho(t,x)=\rho_0\left(\psi^{-1}_t(x)\right)$, where $\psi$ is the flow associated
with $u$; so, $\rho$  satisfies \eqref{eq:ldv}.
Hence, from relation (\ref{eq:cons_en}), we obtain the control of the $L^2$ norm of the velocity field: for all $t\in[0,T_0]$, we have, for some constant $C$ depending only on $\rho_*$ and $\rho^*,$ 
\begin{equation} \label{est:u_L^2}
 \|u(t)\|_{L^2}\,\leq\,C\left(\|u_0\|_{L^2}\,+\,\int^t_0\,\|f(\tau)\|_{L^2}\,d\tau\right)\,.
\end{equation}

{}Next, in the general case where $\div u$ need not be $0,$   applying Proposition \ref{p:transport} yields the following estimates:
\begin{eqnarray}
 \|a(t)\|_{B^{s}_{\infty,r}} & \leq & \|a_0\|_{B^{s}_{\infty,r}}\,
\exp\left(C\int^t_0\|u\|_{B^{s}_{\infty,r}}\,d\tau\right) \label{est:da_Besov} \\
 \|u(t)\|_{B^{s}_{\infty,r}} & \leq & \exp\left(C\int^t_0\|u\|_{B^{s}_{\infty,r}}\,d\tau\right)\,\cdot \,
\biggl(\|u_0\|_{B^s_{\infty,r}}\,+  \label{est:u_Besov} \\
& & +\,\int^t_0e^{-C\int^\tau_0\|u\|_{B^{s}_{\infty,r}}\,d\tau'}
\left(\|f\|_{B^{s}_{\infty,r}}+\|a\|_{B^{s}_{\infty,r}}\|\nabla\Pi\|_{B^s_{\infty,r}}\right)\,d\tau\biggr)\,, \nonumber
\end{eqnarray}
where, in the last line, we have used the fact that $B^s_{\infty,r}$, under  our hypothesis, is an algebra.

\begin{rem} \label{r:a-rho}
Of course,  as $\rho$ and $a$ verify the same equations, they satisfy the same estimates.
\end{rem}

\subsubsection{Estimates for the pressure term}

Let us  use
 the low frequency localization operator $\Delta_{-1}$
 to separate $\nabla\Pi$ into low and high frequencies. We get
$$
\|\nabla\Pi\|_{B^s_{\infty,r}}\,\leq\,\|\Delta_{-1}\nabla\Pi\|_{B^s_{\infty,r}}\,+\,
\|(\mbox{Id}-\Delta_{-1})\nabla\Pi\|_{B^{s}_{\infty,r}}\,.
$$

Observe that $(\Id-\Delta_{-1})\nabla\Pi$ may be computed from 
$\Delta\Pi$ by means of a homogeneous multiplier of degree $-1$ in the sense of 
Proposition \ref{p:CZ}.
Hence
\begin{equation}\label{eq:Pi1}
\|(\mbox{Id}-\Delta_{-1})\nabla\Pi\|_{B^{s}_{\infty,r}}\,\leq\,C\,\|\Delta\Pi\|_{B^{s-1}_{\infty,r}}\,.
\end{equation}

For the low frequencies term, however, the above inequality fails. Now, remembering the definition of $\|\cdot\|_{B^{s}_{\infty,r}}$ and the spectral properties of operator $\Delta_{-1}$, one has that
$$
\|\Delta_{-1}\nabla\Pi\|_{B^{s}_{\infty,r}}\,\leq\,C\,\|\Delta_{-1}\nabla\Pi\|_{L^\infty}\,;
$$
at this point, Bernstein's inequality allows us to  write that
$$
\|\Delta_{-1}\nabla\Pi\|_{B^{s}_{\infty,r}}\,\leq\,C\,\left\|\nabla\Pi\right\|_{L^2}\,.
$$

So  putting together \eqref{eq:Pi1} and the above inequality, we obtain
\begin{equation}\label{eq:Pi}
\|\nabla\Pi\|_{B^s_{\infty,r}}\,\leq\,C\,\left(\|\nabla\Pi\|_{L^2}\,+\,\|\Delta\Pi\|_{B^{s-1}_{\infty,r}}\right)\,.
\end{equation}

First of all, let us see how to control $\|\Delta\Pi\|_{B^{s-1}_{\infty,r}}$.
Recall the third equation of (\ref{eq:ddeuler_a}):
$$
\div\left(a\,\nabla\Pi\right)\,=\,F\quad\hbox{with}\quad
F\,:=\,\div\,(f\,-\,u\cdot\nabla\mc{P}u).$$
Developing the left-hand side of this equation, we obtain
\begin{equation} \label{eq:Lapl_Pi}
 \Delta\Pi\,=\,-\nabla(\log a)\cdot\nabla\Pi\,+\,\frac{F}{a}\cdotp
\end{equation}

Let us consider the first term of the right-hand side of the previous equation.

If $s>1$ then one may use that  $B^{s-1}_{\infty,r}$ is an algebra
and bound  $\|\nabla(\log a)\|_{B^{s-1}_{\infty,r}}$ with
$\|\nabla a\|_{B^{s-1}_{\infty,r}}$ according  to Proposition \ref{p:comp}; we get
$$
\|\nabla(\log a)\cdot\nabla\Pi\|_{B^{s-1}_{\infty,r}}\leq
C\|\nabla a\|_{B^{s-1}_{\infty,r}}\|\nabla\Pi\|_{B^{s-1}_{\infty,r}}.
$$
Now, as  $L^2\hookrightarrow B^{-\frac{N}{2}}_{\infty,\infty}$ (see Corollary \ref{c:embed})
and $B^{s-1}_{\infty,r}$ is an intermediate space between 
$B^{-\frac{N}{2}}_{\infty,\infty}$ and $B^s_{\infty,r},$
standard interpolation inequalities  (see e.g. \cite{BCD}, Chap. 2) ensure that
\begin{equation}\label{est:interpo}
\|\nabla\Pi\|_{B^{s-1}_{\infty,r}}\leq C\|\nabla\Pi\|_{L^2}^\theta
\|\nabla\Pi\|_{B^s_{\infty,r}}^{1-\theta}\quad\hbox{for some }\ 
\theta\in]0,1[.
\end{equation}
Plugging this inequality in \eqref{eq:Pi} and  applying Young's
inequality, we finally obtain
\begin{equation} \label{est:Pi_Bes}
\|\nabla\Pi\|_{B^s_{\infty,r}}\,\leq\,C\,\left(\left(1+\|\nabla a\|^\gamma_{B^{s-1}_{\infty,r}}\right)
\|\nabla\Pi\|_{L^2}\,+\,\left\|\frac{F}{a}\right\|_{B^{s-1}_{\infty,r}}\right)\,,
\end{equation}
where the exponent $\gamma$ depends only on the space dimension $N$ and on $s.$ 
\smallbreak

In the limit case $s=r=1,$ the space
 $B^{s-1}_{\infty,1}$ is no more an algebra and we have to 
 modify the above argument:
  we use the Bony  decomposition \eqref{eq:bony} to write
$$
\nabla(\log a)\cdot\nabla\Pi\,=\,T_{\nabla(\log a)}\nabla\Pi\,+\,T_{\nabla\Pi}\nabla(\log a)\,+\,R(\nabla(\log a),\nabla\Pi)\,.
$$

To estimate first and second term, we can apply Propositions \ref{p:op} and \ref{p:comp}:  we get
\begin{eqnarray} \label{est:T_log-Pi} 
 \|T_{\nabla(\log a)}\nabla\Pi\|_{B^0_{\infty,1}} & \leq & 
 C\,\|\nabla(\log a)\|_{L^\infty}\,\|\nabla\Pi\|_{B^0_{\infty,1}}\\
& \leq & C\,\|\nabla a\|_{L^\infty}\|\nabla\Pi\|_{B^0_{\infty,1}}, \nonumber \\[3ex] \label{est:T_Pi-log} 
\|T_{\nabla\Pi}\nabla(\log a)\|_{B^0_{\infty,1}} & \leq & C\,\|\nabla\Pi\|_{L^\infty}\,\|\nabla(\log a)\|_{B^0_{\infty,1}}\\& \leq & C\,\|\nabla\Pi\|_{L^\infty}\,\|\nabla a\|_{B^0_{\infty,1}}\,. \nonumber
\end{eqnarray}

 A similar inequality is no more true for the remainder term, though. 
 However, one may use that $\nabla\Pi$ is in fact more regular: it belongs to
 $B^{\frac12}_{\infty,1}$ for instance. Hence, using  the embedding $B^{\frac12}_{\infty,1}\hookrightarrow B^0_{\infty,1}$ and Proposition \ref{p:op}, 
 we can write
$$ \begin{array}{lll}
 \|R(\nabla(\log a),\nabla\Pi)\|_{B^0_{\infty,1}} & \leq &
  C\,\|\nabla(\log a)\|_{L^\infty}\,\|\nabla\Pi\|_{B^{\frac12}_{\infty,1}}, \\[1ex]
 & \leq & C\,\|\nabla a\|_{L^\infty}\,\|\nabla\Pi\|_{B^{\frac12}_{\infty,1}}.
 \end{array}
 $$
 
 Putting the above inequality together with \eqref{est:T_log-Pi}
 and \eqref{est:T_Pi-log}, and using that $B^0_{\infty,1}\hookrightarrow L^\infty,$
  we conclude that
 $$
 \|\nabla(\log a)\cdot \nabla\Pi\|_{B^0_{\infty,1}}\leq C\|\nabla a\|_{B^0_{\infty,1}}\|\nabla\Pi\|_{B^{\frac12}_{\infty,1}}.
 $$
 
 Now, using interpolation between Besov spaces, as done for proving
 \eqref{est:interpo}, we get for some suitable $\theta\in]0,1[,$
 $$
 \|\nabla(\log a)\cdot\nabla\Pi\|_{B^{0}_{\infty,1}}\,\leq\,C\,\|\nabla a\|_{B^0_{\infty,1}}\,
\|\nabla\Pi\|^{1-\theta}_{B^{1}_{\infty,1}}\,\|\nabla\Pi\|^{\theta}_{L^2}\,.
$$
Hence   $\|\nabla\Pi\|_{B^1_{\infty,1}}$ satisfies Inequality \eqref{est:Pi_Bes}
for some convenient $\gamma>0.$
\medbreak

Next, let us bound
 the last term of (\ref{eq:Lapl_Pi}). By virtue of Bony's decomposition \eqref{eq:bony}, we have
$$
\frac{F}{a}\,=\,\rho\,F\,=\,T_\rho F\,+\,T_F \rho\,+\,R(\rho,F)\,;
$$
so from Proposition \ref{p:op} we infer that
 \begin{itemize}
 \item $\|T_\rho F\|_{B^{s-1}_{\infty,r}}\,\leq\,C\,\rho^*\,\|F\|_{B^{s-1}_{\infty,r}}$,
 \item $\|T_F \rho\|_{B^{s-1}_{\infty,r}}\,\leq\,C\,\|F\|_{B^{-1}_{\infty,\infty}}\|\rho\|_{B^{s}_{\infty,r}}\,\leq\,
C\,\|F\|_{B^{s-1}_{\infty,r}}\|\rho\|_{B^{s}_{\infty,r}}\,$,
 \item $\|R(\rho,F)\|_{B^{s-1}_{\infty,r}}\,\leq\,\|R(\rho,F)\|_{B^s_{\infty,r}}\,\leq\,C\,\|\rho\|_{B^{1}_{\infty,\infty}}
\|F\|_{B^{s-1}_{\infty,r}}\,\leq\,C\,\|\rho\|_{B^s_{\infty,r}}\|F\|_{B^{s-1}_{\infty,r}}\,$.
\end{itemize}
It is  clear that $\|\div f\|_{B^{s-1}_{\infty,r}}$ can be controlled by $\|f\|_{B^s_{\infty,r}}.$
For the second term of $F$ we have
to take advantage, once again, of Bony's decomposition \eqref{eq:bony} as follows:
$$
\div\,(u\cdot\nabla\mc{P}u)\,=\,\sum_{i,j}\partial_iu^j\,\partial_j\left(\mc{P}u\right)^i\,=\,\sum_{i,j}\left(T_{\partial_iu^j}
\partial_j\mc{P}u^i\,+\,T_{\partial_j\mc{P}u^i}\partial_iu^j\,+\,\partial_iR(u^j,\partial_j\mc{P}u^i)\right)\,,
$$
where in the last equality  we have used also the fact that $\div\cP  u=0.$
Now, for all $i$ and $j$ we have:
\begin{eqnarray*}
 \left\|T_{\partial_iu^j}\partial_j\mc{P}u^i\right\|_{B^{s-1}_{\infty,r}} & \leq & C\,\|\nabla u\|_{L^\infty}\,
\|\nabla\mc{P}u\|_{B^{s-1}_{\infty,r}} \\
 \left\|T_{\partial_j\mc{P}u^i}\partial_iu^j\right\|_{B^{s-1}_{\infty,r}} & \leq & C\,\|\nabla\mc{P}u\|_{L^\infty}
\,\|\nabla u\|_{B^{s-1}_{\infty,r}} \\
 \left\|\partial_iR(u^j,\partial_j\mc{P}u^i)\right\|_{B^{s-1}_{\infty,r}} & \leq & \left\|R(u^j,\partial_j\mc{P}u^i)\right\|_{B^s_{\infty,r}} \\
& \leq & C\,\|u\|_{B^s_{\infty,r}}\,\|\nabla\mc{P}u\|_{B^0_{\infty,\infty}} \\
& \leq & C\,\|u\|_{B^s_{\infty,r}}\,\|\nabla\mc{P}u\|_{B^{s-1}_{\infty,r}}\,.
\end{eqnarray*}
Because, by embedding, 
\begin{equation}\label{est:emb}
\|\nabla\mc{P}u\|_{L^\infty}\,\leq\,C\,\|\nabla\mc{P}u\|_{B^{s-1}_{\infty,r}},
\end{equation}
we thus have
$$
\|\div(u\cdot\nabla\cP u)\|_{B^{s-1}_{\infty,r}}\leq C\|u\|_{B^s_{\infty,r}}
\|\nabla\cP u\|_{B^{s-1}_{\infty,r}}.
$$
In order to bound $\cP u,$ let us decompose it into low and high frequencies as follows:
$$
\cP u=\Delta_{-1}\cP u+(\Id-\Delta_{-1})\cP u.
$$
On the one hand, combining Bernstein's inequality and the
fact that $\cP$ is an orthogonal projector over $L^2$ yields 
$$
\left\|\Delta_{-1}\nabla\mc{P}u\right\|_{L^\infty}\,\leq\,C\,\left\|u\right\|_{L^2}\,.
$$
On the other hand, according to Remark \ref{r:CZ}, one may write that
$$
\|(\Id-\Delta_{-1})\cP u\|_{B^s_{\infty,r}}\leq C\| u\|_{B^s_{\infty,r}}.
$$
Therefore we get
\begin{equation}\label{eq:Pu}
\left\|\nabla\mc{P}u\right\|_{B^{s-1}_{\infty,r}}\,\leq\,C\,\|u\|_{B^s_{\infty,r}\cap L^2}\,,
\end{equation}
from which it follows that
\begin{equation}\label{eq:Fa}
\left\|\frac{F}{a}\right\|_{B^{s-1}_{\infty,r}}\,\leq\,C\,\|a\|_{B^{s}_{\infty,r}}\,\left(\|f\|_{B^{s}_{\infty,r}}+
\|u\|^2_{B^{s}_{\infty,r}\cap L^2}\right)\,.
\end{equation}

It remains us to control $\|\nabla\Pi\|_{L^2}$. Keeping in mind Lemma \ref{l:laxmilgram}, from the third equation of System (\ref{eq:ddeuler_a}) and Inequalities \eqref{est:emb}--\eqref{eq:Pu},
we immediately get
\begin{eqnarray*}
a_*\,\|\nabla\Pi\|_{L^2} & \leq & \|f\|_{L^2}\,+\,\|u\cdot\nabla\mc{P}u\|_{L^2} \\
 & \leq & \|f\|_{L^2}\,+\,\|u\|_{L^2}\,\|\nabla\mc{P}u\|_{L^\infty} \\
& \leq & \|f\|_{L^2}\,+\,C\|u\|^2_{B^s_{\infty,r}\cap L^2}\,.
\end{eqnarray*}

\smallbreak
Putting all these inequalities together, we finally obtain
\begin{eqnarray}
 \|\nabla\Pi\|_{L^1_t(L^2)} & \leq & C\left(\|f\|_{L^1_t(L^2)}\,+\,\int^t_0\|u\|^2_{B^s_{\infty,r}\cap L^2}\,d\tau
\right) \label{est:Pi_t_L2} \\
\|\nabla\Pi\|_{L^1_t(B^s_{\infty,r})} & \leq & C\left(\left(1+\|\nabla a\|^\gamma_{L^\infty_t(B^{s-1}_{\infty,r})}\right)
\|\nabla\Pi\|_{L^1_t(L^2)}\,+\right. \label{est:Pi_t_B} \\
 & & \qquad\left.+\,\|a\|_{L^\infty_t(B^{s}_{\infty,r})}
\left(\|f\|_{L^1_t(B^{s}_{\infty,r})}+\int^t_0\|u\|^2_{B^{s}_{\infty,r}\cap L^2}d\tau\right)\right)\,. \nonumber
\end{eqnarray}

\subsubsection{Final estimate}

First of all, let us fix $T>0$ so small as to satisfy
\begin{equation}\label{eq:smallu}
\exp\biggl(C\int^T_0\|u\|_{B^{s}_{\infty,r}}\,dt\biggr)\,\leq\,2\,,
\end{equation}
a fact that is always possible because of the continuity of $u$ with respect the time variable.
\smallbreak
Then, setting
\begin{eqnarray*}
U(t) & := & \|u(t)\|_{L^2\cap B^{s}_{\infty,r}}=\|u(t)\|_{L^2}+\|u(t)\|_{B^{s}_{\infty,r}} \\
U_0(t) & := & \|u_0\|_{L^2\cap B^{s}_{\infty,r}}\,+\,\int^t_0\|f\|_{L^2\cap B^{s}_{\infty,r}}\,d\tau\,
\end{eqnarray*}
and combining estimates (\ref{est:u_L^2}), (\ref{est:da_Besov}), (\ref{est:u_Besov}), (\ref{est:Pi_t_L2}) and (\ref{est:Pi_t_B}),
 we get
\begin{equation}
U(t)\,\leq\,C\left(U_0(t)\,+\,\int^t_0U^2(\tau)\,d\tau\right)\quad\hbox{for all }\  t\in[0,T],\label{est:U(t)}
\end{equation}
where the constant $C$ depends only   on $s,$ $N,$ $\|a_0\|_{B^s_{\infty,r}}$, $a_*$ and $a^*$.

So, taking $T$ small enough  and changing once more the multiplying constant if needed, 
a standard bootstrap argument allows to show that
$$
U(t)\,\leq\,C\,U_0(t)\qquad\qquad\qquad\qquad\forall\,t\,\in\,[0,T]\,.
$$


\subsection{Existence of a solution to \eqref{eq:ddeuler_a}}
We proceed in  two  steps: first we  construct
inductively a sequence of smooth global approximate solutions, defined as solutions of a linear system, and then we  prove the convergence
of this sequence to a solution of the nonlinear system \eqref{eq:ddeuler_a} with the required property.
Recall that to simplify the presentation we have assumed that $T_0=+\infty$
and that we focus on the evolution for positive times.

\subsubsection{Construction of the sequence of approximate solutions}

First, we smooth out the data (by convolution for instance) so as to get
a sequence $(a_0^n,u_0^n,f^n)_{n\in\N}$
 such that  $u_0^n\in H^\infty,$
$f^n\in\cC(\R_+;H^\infty),$
 $a_0^n$ and its derivatives at any order are bounded and 
 \begin{equation}\label{eq:ldv0}
a_*\leq a_0^n\leq a^*,
\end{equation}
with in addition
\begin{itemize}
\item $a_0^n\rightarrow a_0$ in $B^s_{\infty,r},$
\item $u_0^n\rightarrow u_0$ in $L^2\cap B^s_{\infty,r},$
\item $f^n\rightarrow f$ in $\cC(\R_+;L^2)\cap L^1(\R_+;B^s_{\infty,r}).$
\end{itemize}

In order to construct  a sequence of smooth approximate solutions, we  argue by induction. We first set $a^0=a_0^0$, $u^0=u_0^0$ and $\nabla\Pi^0=0$.

Now, suppose we have already built a smooth approximate
solution $(a^n,u^n,\nabla\Pi^n)$ over $\R_+\times\R^N$
with $a^n$ satisfying \eqref{eq:ldv}.  
In order to  construct the $(n+1)$-th  term of the sequence, 
we first 
define $a^{n+1}$ to be the solution of the linear transport equation
$$
\partial_t a^{n+1}\,+\,u^n\cdot\nabla a^{n+1}\,=\,0\,
$$
with initial datum $a^{n+1}|_{t=0}=a_0^{n+1}.$

Given that $u^n$ is smooth, 
its flow is smooth too so that 
 $a^{n+1}(t,x)=a_0^{n+1}\left((\psi^n_t)^{-1}(x)\right)$, where $\psi^n_t$ is the flow at time $t.$
 Note that $\psi^n_t$ is a smooth diffeomorphism on the whole  $\mbb{R}^N$. {}From  this fact, we  gather that $a^{n+1}$ is smooth and satisfies \eqref{eq:ldv}. 
 Furthermore, by virtue of Proposition \ref{p:transport}, 
\begin{equation} \label{est:a^n+1_B}
 \|a^{n+1}(t)\|_{B^{s}_{\infty,r}}\,\leq\,\|a_0^{n+1}\|_{B^{s}_{\infty,r}}\,
\exp\left(C\int^t_0\|u^n\|_{B^s_{\infty,r}}\,d\tau\right)\,.
\end{equation}

Note that the reciprocal function $\rho^{n+1}$ of $a^{n+1}$
satisfies  $\rho^{n+1}(t,x)=\rho^{n+1}_0\left(\left(\psi^n_t\right)^{-1}(x)\right)$,
together with \eqref{eq:ldv} and  the equation
$$
\partial_t\rho^{n+1}\,+\,u^n\cdot\nabla\rho^{n+1}\,=\,0\,.
$$
Hence it also fulfills Inequality \eqref{est:a^n+1_B} up to a change of $a_0^{n+1}$ in $\rho_0^{n+1}.$

\smallbreak
At this point, we define $u^{n+1}$ as the unique  smooth solution of the transport equation
$$
\left\{ \begin{array}{l}
  \partial_tu^{n+1}\,+\,u^n\cdot\nabla u^{n+1}\,=\,f^{n+1}\,-\,a^{n+1}\,\nabla\Pi^n \\[1ex]
 u^{n+1}|_{t=0}\,=\,u_0^{n+1}\,.
\end{array} \right.
$$

Since the right-hand side belongs to $L^1_{loc}(\mbb{R}_+;L^2)$, from classical results
for transport equation we get that $u^{n+1}\in\cC(\mbb{R}_+;L^2)$. Besides, as $\rho^n=(a^n)^{-1}$ for all $n$, if we differentiate with respect
to time the product $\sqrt{\rho^{n+1}}u^{n+1}$ and take the scalar product with $u^{n+1}$, we obtain
$$
\frac{1}{2}\frac{d}{dt}\left\|\sqrt{\rho^{n+1}}u^{n+1}\right\|^2_{L^2}\,=\,\frac{1}{2}\int\rho^{n+1}|u^{n+1}|^2\,\div u^n\,dx\,+\,
\int\rho^{n+1}u^{n+1}\cdot f^{n+1}\,dx\,-\,\int\nabla\Pi^n\cdot u^{n+1}\,dx.
$$
Observe that $u^n$ and $u^{n+1}$ need not be divergence free; nevertheless one may control
$\|\div u^n\|_{L^\infty}$ with $\|u^n\|_{B^s_{\infty,r}}$. So, from the previous equality, applying Gronwall's Lemma, it is easy to see that
\begin{equation} \label{est:u^n+1_L^2}
\left\|\sqrt{\rho^{n+1}(t)}u^{n+1}(t)\right\|_{L^2}\leq\left\|\sqrt{\rho^{n+1}_0}u^{n+1}_0\right\|_{L^2}
+C\int^t_0\left(\|f^{n+1}\|_{L^2}\,+\,
\|\nabla\Pi^n\|_{L^2}\,+\|u^n\|_{B^s_{\infty,r}}\right)d\tau\,.
\end{equation}

\smallbreak
Finally, we have to define the approximate pressure $\Pi^{n+1}$. We have already proved that $a^{n+1}$ satisfies the ellipticity hypothesis \eqref{eq:ldv}; so
we can consider the elliptic equation
$$
\div\left(a^{n+1}\,\nabla\Pi^{n+1}\right)\,=\,\div\left(f^{n+1}\,-\,u^{n+1}\cdot\nabla\mc{P}u^{n+1}\right)\,.
$$
As $f^{n+1}$ and $u^{n+1}$ are  in $\cC(\R_+;H^\infty),$
the classical theory for elliptic equations ensures
that the above equation has a unique solution $\nabla\Pi^{n+1}$ in $\cC(\mbb{R}_+;H^\infty).$
In addition, going along the lines of the proof of \eqref{est:Pi_t_L2}, we get 
\begin{equation} \label{est:Pi^n+1_L^2}
\left\|\nabla\Pi^{n+1}\right\|_{L^1_t(L^2)}\,\leq\, C\left(\|f^{n+1}\|_{L^1_t(L^2)}\,+\,\int^t_0
\|u^{n+1}\|_{B^s_{\infty,r}\cap L^2}^2\,d\tau\right)\,.
\end{equation}

Of course, by embedding, we have  $\nabla\Pi^{n+1}\in\cC(\mbb{R}_+;B^s_{\infty,r}).$
Hence, arguing as for proving \eqref{est:Pi_t_B}, we get 
\begin{eqnarray} \label{est:Pi^n+1_B}
\left\|\nabla\Pi^{n+1}\right\|_{L^1_t(B^s_{\infty,r})} & \leq &  C\|a^{n+1}\|_{L_t^\infty(B^s_{\infty,r})}
\left(\|f^{n+1}\|_{L^1_t(B^s_{\infty,r})}+
\int^t_0\|u^{n+1}\|^2_{B^s_{\infty,r}\cap L^2}d\tau\right) \\
& & \qquad+\,
C\left(1+\|\nabla a^{n+1}\|^\gamma_{L^\infty_t(B^{s-1}_{\infty,r})}\right)\|\nabla\Pi^{n+1}\|_{L^1_t(L^2)}. \nonumber
\end{eqnarray}

Note also that the norms of the approximate data that we use in \eqref{est:a^n+1_B},
\eqref{est:u^n+1_L^2}, \eqref{est:Pi^n+1_L^2} and \eqref{est:Pi^n+1_B}
may be bounded \emph{independently of $n$}. 
Therefore, repeating the arguments leading to \eqref{est:U(t)} and to Theorem 1 of \cite{D2}, 
one may find some positive time $T$ which may depend  on  
$\|\rho_0\|_{B^s_{\infty,r}}$, $\|u_0\|_{B^s_{\infty,r}\cap L^2}$
and $\|f\|_{L^1([0,T];B^s_{\infty,r}\cap L^2)}$ but is \emph{independent of $n$}
such that \begin{itemize}
\item $(a^n)_{n\in\N}$ is bounded in $L^\infty([0,T];B^s_{\infty,r}),$
\item $(u^n)_{n\in\N}$ is bounded in $L^\infty([0,T];B^s_{\infty,r}\cap L^2),$
\item $(\nabla\Pi^n)_{n\in\N}$ is bounded in $L^1([0,T];B^s_{\infty,r})\cap L^\infty([0,T];L^2).$
\end{itemize}


\subsubsection{Convergence of the sequence}

Let us observe that the function  $\wtilde{a}^n:\,=\,a^n-a^n_0$  satisfies
$$
\left\{\begin{array}{lcl}
        \partial_t\wtilde{a}^n & = & -\,u^{n-1}\cdot\nabla a^n \\[1ex]
       \wtilde{a}^n|_{t=0} & = & 0\,.
       \end{array}
\right.
$$
Because   $u^{n-1}\in\cC([0,T];L^2)$ and $\nabla a^n\in\cC_b([0,T]\times\R^N),$
it immediately follows that $\wtilde{a}^n\in\cC^1([0,T];L^2).$
Now we want to prove that the sequence $(\wtilde{a}^n,u^n,\nabla\Pi^n)_{n\in\N}$, built in this way, is a Cauchy sequence in $\cC([0,T];L^2)$.
So let us  define
\begin{eqnarray*}
  \delta\!  a^n & := & {a}^{n+1}\,-\,{a}^n\,, \\
  \delta\! \tilde a^n & := & \wtilde{a}^{n+1}\,-\,\wtilde{a}^n \; =\; \delta\! a^n\,-\,\delta\! a^n_0\,, \\
 \delta\!\rho^n&:=&\rho^{n+1}\,-\,\rho^n\,,\\
\delta\! u^n & := & u^{n+1}\,-\,u^n\,, \\
\delta\!\Pi^n & := & \Pi^{n+1}\,-\,\Pi^n\,, \\
\delta\! f^n & := & f^{n+1}\,-\,f^n\,.
\end{eqnarray*}

Let us emphasize that, by assumption  and embedding, we have
\begin{itemize}
\item $a_0^n\rightarrow a_0$ in $C^{0,1},$
\item $u_0^n\rightarrow u_0$ in $L^2,$
\item $f^n\rightarrow f$ in $L^1([0,T];L^2).$ 
\end{itemize}
This will be the key to our proof of convergence. 
\medbreak
Let us first focus on $\tilde a^n.$
By construction,  $\delta\! \tilde a^n$ belongs to $\cC^1([0,T];L^2)$ and satisfies the equation
$$
\partial_t\delta\! \tilde a^n\,=\,-u^{n}\cdot\nabla\delta\! \tilde a^n\,-\,\delta\! u^{n-1}\cdot\nabla a^n
-u^n\cdot\nabla\delta\!a_0^n
$$
from which, taking the scalar product in $L^2$ with $\delta\! \tilde a^n$, we obtain
$$
\frac{1}{2}\frac{d}{dt}\left\|\delta\! \tilde a^n\right\|^2_{L^2}\,
=\,\frac{1}{2}\int\left(\delta\! \tilde a^n\right)^2\div u^n\,dx\,-\,
\int\delta\! u^{n-1}\cdot\nabla a^n\,\delta\! \tilde a^n\,dx-\int u^n\cdot\nabla\delta\!a_0^n\,\delta\!\tilde{a}^n\,dx\,.
$$
So, keeping in mind that $\delta\! \tilde a^n(0)=0$ and integrating with respect to the time variable one has
\begin{equation} \label{est:d-a^n}
 \left\|\delta\! \tilde a^n(t)\right\|_{L^2}\,\leq\,\int^t_0\biggl(\frac{1}{2}\left\|\div u^n\right\|_{L^\infty}
\left\|\delta\! \tilde a^{n}\right\|_{L^2}+
\|\nabla a^n\|_{L^\infty}\|\delta\! u^{n-1}\|_{L^2}+\|u^n\|_{L^2}\|\nabla\da_0^n\|_{L^\infty}\biggr)\,d\tau\,.
\end{equation}

Equally easily, one can see that the following equality holds true:
$$
\rho^{n+1}\left(\partial_t\delta\! u^n\,+\,u^n\cdot\nabla\delta\! u^n\right)\,+\,\nabla\delta\!\Pi^{n-1}\,=\,
\rho^{n+1}\left(\delta\!f^n-\,\delta\! u^{n-1}\cdot\nabla u^n\,-\,\delta\! a^n\,\nabla\Pi^{n-1}\right)\,;
$$
taking the scalar product in $L^2$  with $\delta\! u^n$, integrating by parts, remembering the first equation of (\ref{eq:ddeuler})
at $(n+1)$-th step, we finally get
$$\displaylines{
\left\|\sqrt{\rho^{n+1}(t)}\delta\! u^n(t)\right\|_{L^2}  \leq  \int^t_0\left\|\div u^n\right\|_{L^\infty}\,
\left\|\sqrt{\rho^{n+1}}\delta\! u^n\right\|_{L^2}\,d\tau\hfill\cr\hfill+\int^t_0\biggl(\left\|\nabla u^n\right\|_{L^\infty}
\left\|\sqrt{\rho^{n+1}}\delta\! u^{n-1}\right\|_{L^2}
+\left\|\sqrt{\rho^{n+1}}\nabla\Pi^{n-1}\right\|_{L^\infty}\left\|\delta\!\tilde a^n\right\|_{L^2}\hfill\cr\hfill+
\left\|\sqrt{\rho^{n+1}}\nabla\Pi^{n-1}\right\|_{L^2}\left\|\delta\! a^n_0\right\|_{L^\infty}+
\left\|\frac{\nabla\delta\!\Pi^{n-1}}{\sqrt{\rho^{n+1}}}\right\|_{L^2}+\sqrt{\rho^*}\|\delta\! f^n\|_{L^2}\biggr)d\tau\,.}
$$
{}From (\ref{est:d-a^n}), Gronwall's Lemma and \eqref{eq:ldv0},  we thus get
for some constant $C$ depending only on $a_*$ and $a^*,$
$$\displaylines{
 \left\|(\delta\! \tilde a^n,\delta\! u^n)(t)\right\|_{L^2} \leq C\biggl(e^{A^n(t)}\|\du_0^n\|_{L^2}
+\int^t_0e^{A^n(t)-A^n(\tau)}\biggl(\|(\nabla a^n,\nabla u^n)\|_{L^\infty}
\left\|\delta\! u^{n-1}\right\|_{L^2} \hfill\cr\hfill+
\left\|\nabla\delta\!\Pi^{n-1}\right\|_{L^2}+\left\|\nabla\Pi^{n-1}\right\|_{L^2}\left\|\delta\! a^n_0\right\|_{L^\infty}+\|u^n\|_{L^2}\|\nabla\da_0^n\|_{L^\infty}+\|\delta\!f^n\|_{L^2}\biggr)\,d\tau\,,} 
$$
where we have set
$$
A^n(t)\,:=\,\int^t_0\left(\left\|\div u^n\right\|_{L^\infty}\,+\,
\left\|\sqrt{\rho^{n+1}}\nabla\Pi^{n-1}\right\|_{L^\infty}\right)d\tau\,.
$$
Of course,  the
 uniform a priori estimates of the previous step  allow us to control the exponential term for all $t\in[0,T]$
by some constant $C_T.$ 
\smallbreak

Next, we have to deal with the term $\nabla\delta\!\Pi^{n-1}$. We notice that it satisfies the elliptic equation
$$
-\div\left(a^{n-1}\nabla\delta\!\Pi^{n-1}\right)\,=\,\div\left(-\delta\! a^{n-1}\nabla\Pi^n\,-\,u^{n-1}\cdot\nabla\mc{P}\delta\! u^{n-1}\,
-\,\delta\! u^{n-1}\cdot\nabla\mc{P}u^n\,+\,\delta\! f^{n-1}\right)\,.
$$
Then applying the following algebraic identity 
$$
\div(v\cdot\nabla w)\,=\,\div(w\cdot\nabla v)\,+\,\div(v\,\div w)\,-\,\div(w\,\div v)
$$
to $v=u^{n-1}$ and $w=\mc{P}\delta\! u^{n-1}$, and remembering that $\div\mc{P}\du^{n-1}=0$, we get
\begin{eqnarray*}
\div\left(a^{n-1}\nabla\delta\!\Pi^{n-1}\right) &  = & \div\biggl(\mc{P}\delta\! u^{n-1}\,\div u^{n-1}\,-\,
\mc{P}\delta\! u^{n-1}\cdot\nabla u^{n-1}\,-\,\delta\! u^{n-1}\cdot\nabla\mc{P}u^n \\
 & & \,-\,\delta\! a^{n-1}\nabla\Pi^n\,+\,\delta\! f^{n-1}\biggr)\,.
\end{eqnarray*}
Therefore, from Lemma \ref{l:laxmilgram} and the fact that $\|\mc{P}\|_{\mc{L}(L^2;L^2)}=1$, one immediately has the following inequality:
\begin{eqnarray} \label{est:d-Pi}
a_*\left\|\nabla\delta\!\Pi^{n-1}\right\|_{L^2} & \leq &
 \left\|\delta\! \tilde a^{n-1}\right\|_{L^2}\left\|\nabla\Pi^n\right\|_{L^\infty}\,+ \,\left\|\delta\! a^{n-1}_0\right\|_{L^\infty}\left\|\nabla\Pi^n\right\|_{L^2}\,+\,\|\delta\! f^{n-1}\|_{L^2} \\
& & +\,\left\|\delta\! u^{n-1}\right\|_{L^2}\left(\left\|\div u^{n-1}\right\|_{L^\infty}\,+\,
\left\|\nabla u^{n-1}\right\|_{L^\infty}\,+\,\left\|\nabla\mc{P}u^n\right\|_{L^\infty}\right)\,.\nonumber
\end{eqnarray}
Due to a priori estimates, we finally obtain, for all $t\in[0,T]$,
\begin{eqnarray*}
 \left\|(\delta\! \tilde a^n,\delta\! u^n)(t)\right\|_{L^2} & \leq & C_{T}\biggl(\|\du^n_0\|_{L^2}
+\int^t_0\Bigl(\left\|(\delta\! a^{n-1},\delta\! u^{n-1})\right\|_{L^2}\\&&\,+\,\left\|\nabla\delta\!\Pi^{n-1}\right\|_{L^2}+\|\da_0^n\|_{C^{0,1}}+\|\delta\!f^n\|_{L^2}\Bigr)\,d\tau\biggr) \\
 \left\|\nabla\delta\!\Pi^{n-1}\right\|_{L^2} & \leq & C_{T}\left(\left\|\delta\! \tilde a^{n-1}\right\|_{L^2}\,+\,\left\|\delta\! u^{n-1}\right\|_{L^2}\,+\,\|\da_0^{n-1}\|_{L^\infty}\,+\,\|\delta\! f^{n-1}\|_{L^2}\right)\,;
\end{eqnarray*}
so,  plugging the second inequality in the first one, we find out that for all  $t\in[0,T],$
\begin{equation} \label{est:Cauchy}
 \left\|(\delta\! \tilde a^n,\delta\! u^n)(t)\right\|_{L^2} \leq
\eps_n +C_T\int^t_0\left\|(\delta\! \tilde a^{n-1},\delta\! u^{n-1})\right\|_{L^2}\,d\tau
\end{equation}
with
$$
\eps_n:= C_T\biggl(\|\du^n_0\|_{L^2}+\int_0^T\bigl(\|\delta\!f^{n-1}\|_{L^2}+\|\delta\!f^n\|_{L^2}
+\|\da_0^{n-1}\|_{L^\infty}+\|\nabla\da_0^{n-1}\|_{L^\infty}\bigr)\,dt\biggr).
$$

Now, we have
$$
\sum_n \eps_n<\infty.
$$
{}From this and \eqref{est:Cauchy}, it is easy 
to conclude that 
$$
\sum_n\sup_{t\in[0,T]}\bigl(\|\delta\!\tilde a^n(t)\|_{L^2}
+\|\du^n(t)\|_{L^2}\bigr)<\infty.
$$
In other words,  $(\tilde a^n)_{n\in\N}$
and $(u^n)_{n\in\N}$ are 
 Cauchy sequences in $\cC([0,T];L^2)$; therefore they converge to
some functions $\wtilde{a}$, $u\,\in\cC([0,T];L^2)$. In the same way, it is  clear that $\left(\nabla\Pi^n\right)_{n\in\mbb{N}}$ converges to
some $\nabla\Pi\in\cC([0,T];L^2)$.

\smallbreak
Defining  $a\,:=\,\wtilde{a}+a_0$, it remains to show that $a$, $u$ and $\nabla\Pi$ are indeed solutions of the initial system. We already know that
$a$, $u$ and $\nabla\Pi\;\in\cC([0,T];L^2)$. In addition, 
\begin{itemize}
 \item thanks to Fatou's  property in  Besov spaces, as $\left(a^n\right)_{n\in\mbb{N}}$ is bounded in $L^\infty([0,T];B^{s}_{\infty,r})$, we obtain that
$a\in L^\infty([0,T];B^s_{\infty,r})$ and satisfies \eqref{eq:ldv0};
 \item in the same way, $u\in L^\infty([0,T];B^s_{\infty,r})$ because also $\left(u^n\right)_{n\in\mbb{N}}$ is bounded in the same space;
 \item finally, $\nabla\Pi\in L^1([0,T];B^s_{\infty,r})$ because the sequence $\left(\nabla\Pi^n\right)_{n\in\mbb{N}}$ is bounded in the same
functional space.
\end{itemize}
By interpolation we get that the sequences converge strongly to the solutions in every intermediate space between $\cC([0,T];L^2)$ and
$\cC([0,T];B^s_{\infty,r})$, that is enough to pass to the limit in the equations satisfied by $\left(a^n,u^n,\nabla\Pi^n\right)$.
So, $\left(a,u,\nabla\Pi\right)$ satisfies System \eqref{eq:ddeuler_a}.

Finally, continuity properties of the solutions with respect to the time variable can be recovered from the equations satisfied by them, using
classical results for transport equations in Besov spaces (see Proposition \ref{p:transport}).


\subsection{Uniqueness of the solution}

Uniqueness of the solution to System \eqref{eq:ddeuler} is a straightforward consequence of the following stability result, the proof
of which can be found in \cite{D2}.

\begin{prop} \label{p:uni}
 Let $(\rho_1,u_1,\nabla\Pi_1)$ and $(\rho_2,u_2,\nabla\Pi_2)$ satisfy System \eqref{eq:ddeuler} with external forces 
$f_1$ and $f_2$, respectively.
Suppose that $\rho_1$ and $\rho_2$ both satisfy \eqref{eq:ldv}.
Assume also that:
\begin{itemize}
 \item $\delta\!\rho:=\rho_2-\rho_1$ and $\delta\! u:=u_2-u_1$ both belong to $\cC^1([0,T];L^2)$,
 \item $\delta\! f:=f_2-f_1\,\in\cC([0,T];L^2)$,
 \item $\nabla\rho_1$, $\nabla u_1$ and $\nabla\Pi_1$ belong to $L^1([0,T];L^\infty)$.
\end{itemize}

Then for all $t\in[0,T]$ we have
$$
e^{-A(t)}\left(\left\|\delta\!\rho(t)\right\|_{L^2}+\left\|\left(\sqrt{\rho_2}\delta\! u\right)(t)\right\|_{L^2}\right) \leq 
\left\|\delta\!\rho(0)\right\|_{L^2}+\left\|\left(\sqrt{\rho_2}\delta\! u\right)(0)\right\|_{L^2}+ 
 \int^t_0 e^{-A(\tau)}\left\|\left(\sqrt{\rho_2}\delta\! f\right)\right\|_{L^2}\,d\tau
 $$
with
$$
A(t):=\int^t_0\left(\left\|\frac{\nabla\rho_1}{\sqrt{\rho_2}}\right\|_{L^\infty}+
\left\|\frac{\nabla\Pi_1}{\rho_1\sqrt{\rho_2}}\right\|_{L^\infty}+
\left\|\nabla u_1\right\|_{L^\infty}\right)d\tau\,.
$$
\end{prop}

\smallbreak
\begin{proof}[Proof of uniqueness in Theorem \ref{th:L^2}]
 Let us suppose that there exist two solutions $(\rho_1,u_1,\nabla\Pi_1)$ and $(\rho_2,u_2,\nabla\Pi_2)$ to System \eqref{eq:ddeuler}
 corresponding to the same data and satisfying the hypotheses of Theorem \ref{th:L^2}.
Then, as one can easily verify,  these solutions satisfy 
the assumptions of Proposition \ref{p:uni}.
For instance, that  $\delta\!\rho\in\cC^1([0,T];L^2)$ is an immediate consequence of the fact that, for $i=1, 2$, the velocity
field $u_i$ is in $\cC([0,T];L^2)$ and $\nabla\rho_i$ is in $\cC([0,T];L^\infty)$, so that $\partial_t\rho_i\in\cC([0,T];L^2)$.

So, Proposition \ref{p:uni}  implies that $(\rho_1,u_1,\nabla\Pi_1)\equiv(\rho_2,u_2,\nabla\Pi_2)$.
\end{proof}


\subsection{Proof of the continuation criterion}

Now, we want to prove the continuation criterion  for the solution to (\ref{eq:ddeuler}). We 
proceed  in two steps. As usual, we will suppose
Condition $(C)$ to be satisfied with $p=\infty$.
The first step of the proof is given by the following lemma.

\begin{lem} \label{l:cont_1} 
 Let $(\rho,u,\nabla\Pi)$ be a solution of System \eqref{eq:ddeuler} on $[0,T^*[\times\mbb{R}^N$
such that\footnote{with the usual convention that continuity in time is weak if $r=\infty$.}
\begin{itemize}
 \item $u\in\cC([0,T^*[;B^s_{\infty,r})\cap\cC^1([0,T^*[;L^2)$,
 \item $\rho\in\cC([0,T^*[;B^s_{\infty,r})$ and satisfies \eqref{eq:ldv}.
\end{itemize}

Suppose also that Condition \eqref{eq:cont_cond} holds and that $T^*$ is finite. Then
$$
\sup_{t\in[0,T^*[}\left(\|u(t)\|_{B^s_{\infty,r}\cap L^2}\,+\,\|\rho(t)\|_{B^{s}_{\infty,r}}\right)\,+\,
\int^{T^*}_0\|\nabla\Pi\|_{B^s_{\infty,r}}\,dt\,<\,\infty\,.
$$
\end{lem}

\begin{proof}[Proof of Lemma \ref{l:cont_1}]
It is only a matter of repeating  the a priori estimates of the previous section, but in a more accurate way. Note that $a:=1/\rho$ satisfies the same
hypothesis as $\rho$, so we will work without distinction with these two quantities, according to what is more convenient to us, and set 
  $q=\rho$ or $a.$ Recall that
$$
\d_tq+u\cdot\nabla q=0.
$$
Hence, applying operator $\Delta_j$ yields
$$
\d_t\Delta_jq+u\cdot\nabla\Delta_jq=[u\cdot\nabla,\Delta_j]q
$$
whence, for all $t\in[0,T^*[,$ 
\begin{equation}\label{eq:q}
2^{js}\|\Delta_jq(t)\|_{L^\infty}\leq2^{js}\|\Delta_jq_0\|_{L^\infty}
+\int_0^t2^{js}\|[u\cdot\nabla,\Delta_j]q\|_{L^\infty}\,d\tau.
\end{equation}
Now, Lemma 2.100 in \cite{BCD} ensures that 
$$
\Bigl\|\bigl(2^{js}\|[u\cdot\nabla,\Delta_j]q\|_{L^\infty}\bigr)_j\Bigr\|_{\ell^r}
\leq C\bigl(\|\nabla u\|_{L^\infty}\|q\|_{B^s_{\infty,r}}+\|\nabla q\|_{L^\infty}\|\nabla u\|_{B^{s-1}_{\infty,r}}\bigr).
$$
Hence, performing an $\ell^r$ summation in \eqref{eq:q}, we get
\begin{equation}\label{eq:q1}
\|q(t)\|_{B^{s}_{\infty,r}}\,\leq\,\|q_0\|_{B^{s}_{\infty,r}}\,+\,C\int_0^t\bigl(\|\nabla u\|_{L^\infty}\|q\|_{B^s_{\infty,r}}
+\|\nabla q\|_{L^\infty}\|u\|_{B^s_{\infty,r}}\bigr)\,d\tau.
\end{equation}
As regards the velocity field, we have according to \eqref{est:u_L^2}, 
$$
\|u(t)\|_{L^2} \leq  C\left(\|u_0\|_{L^2}\,+\,\int^t_0\|f\|_{L^2}\,d\tau\right)\,,
$$
and the last part of Proposition \ref{p:transport} guarantees that
$$\displaylines{
 \|u(t)\|_{B^{s}_{\infty,r}}  \leq  \exp\left(C\int^t_0\|\nabla u\|_{L^\infty}\,d\tau\right)\hfill\cr\hfill\times
\left(\|u_0\|_{B^s_{\infty,r}}\,+\,\int^t_0e^{-C\int^\tau_0\|\nabla u\|_{L^\infty}\,d\tau'}
\left(\|f\|_{B^{s}_{\infty,r}}+\|a\nabla\Pi\|_{B^s_{\infty,r}}\right)\,d\tau\right).}
$$
Bounding the last term according to Corollary \ref{c:op}, we thus get
$$\displaylines{
 \|u(t)\|_{B^{s}_{\infty,r}}  \leq  \exp\left(C\int^t_0\|\nabla u\|_{L^\infty}\,d\tau\right)\hfill\cr\hfill\times
\left(\|u_0\|_{B^s_{\infty,r}}\,+\,\int^t_0e^{-C\int^\tau_0\|\nabla u\|_{L^\infty}\,d\tau'}
\left(\|f\|_{B^{s}_{\infty,r}}+a^*\|\nabla\Pi\|_{B^s_{\infty,r}}+\|\nabla a\|_{B^{s-1}_{\infty,r}}\|\nabla\Pi\|_{L^\infty}\right)\,d\tau\right).}
$$
As regards the pressure term, we have
\begin{eqnarray*}
\|\nabla\Pi\|_{L^2} & \leq & C\left(\|f\|_{L^2}\,+\,\|u\|_{L^2}\,\|\nabla u\|_{L^\infty}\right) \\
\|\nabla\Pi\|_{B^s_{\infty,r}} & \leq & C\left(\|\nabla\Pi\|_{L^2}\,+\,
\|\nabla a\cdot\nabla\Pi\|_{B^{s-1}_{\infty,r}}\,+\,
\left\|\frac{1}{a}\,\div\left(f-u\cdot\nabla u\right)\right\|_{B^{s-1}_{\infty,r}}\right).
\end{eqnarray*}

Note that  Bony's decomposition combined with the fact that $\div u=0$ ensures that
$$
\|\div(u\cdot\nabla u)\|_{B^{s-1}_{\infty,r}}\leq C\|\nabla u\|_{L^\infty}\|u\|_{B^s_{\infty,r}}.
$$
In addition, \emph{under the assumption that $s>1$}, Corollary \ref{c:op} implies that
\begin{equation}\label{eq:dPi}
\|\nabla a\cdot\nabla\Pi\|_{B^{s-1}_{\infty,r}}\leq 
C\bigl(\|\nabla a\|_{L^\infty}\|\nabla\Pi\|_{B^{s-1}_{\infty,r}}\,+\,
\|\nabla a\|_{B^{s-1}_{\infty,r}}\|\nabla\Pi\|_{L^\infty}\bigr).
\end{equation}
So finally
$$
\displaylines{\|\nabla\Pi\|_{B^s_{\infty,r}}\leq C\biggl(\|\nabla\Pi\|_{L^2}\,+\,\|\nabla a\|_{L^\infty}\|\nabla\Pi\|_{B^{s-1}_{\infty,r}}\,+\,
\|\nabla a\|_{B^{s-1}_{\infty,r}}\|\nabla\Pi\|_{L^\infty}\hfill\cr\hfill
+\|a\|_{B^{s}_{\infty,r}}\left(\|f\|_{B^s_{\infty,r}}+\|\nabla u\|_{L^\infty}\|u\|_{B^{s}_{\infty,r}}\right)\biggr).}
$$

Putting together all these estimates and applying Gronwall's Lemma, we obtain if $s>1,$
\begin{eqnarray*}
 \|\nabla a\|_{B^{s-1}_{\infty,r}}+\|u(t)\|_{B^s_{\infty,r}\cap L^2} & \leq & \,\exp\left(C\int^t_0\|(\nabla a,\nabla u,\nabla\Pi)\|_{L^\infty}d\tau
\right)\left(\|\nabla a_0\|_{B^{s-1}_{\infty,r}}+\right. \\
& & \left.\|u_0\|_{B^s_{\infty,r}\cap L^2}+\|f\|_{B^{s}_{\infty,r}\cap L^2}+
\int^t_0\|\nabla a\|_{L^\infty}\|\nabla\Pi\|_{B^{s-1}_{\infty,r}}\,d\tau\right)\,,
\end{eqnarray*}
where the constant $C$ depends  only on $s,$ $a_*$, $a^*$ and $N$.

Now, the equation for $\nabla a$ and Gronwall inequality immediately ensure that
\begin{equation}\label{eq:Da}
\|\nabla a(t)\|_{L^\infty}\,\leq\,\|\nabla a_0\|_{L^\infty}\,\exp\left(\int^t_0\|\nabla u\|_{L^\infty}\,d\tau\right)\,,
\end{equation}
which, thanks to Hypothesis (\ref{eq:cont_cond}) implies  that $\nabla a$ is bounded in time with values in $L^\infty$.

Moreover, by hypothesis $\nabla\Pi\in L^1([0,T^*[;B^{s-1}_{\infty,r})$ and $\nabla u\in L^1([0,T^*[;L^\infty)$; at this point, keeping in mind the
embedding $B^{s-1}_{\infty,r}\hookrightarrow L^\infty$, the previous inequality gives us the thesis of the lemma in the case $s>1.$

In the endpoint case $s=r=1,$  Inequality \eqref{eq:dPi} fails.
In order to complete the proof of the lemma, we will have 
to take advantage of the vorticity equation associated to
\eqref{eq:ddeuler}. This is  postponed to the next section.
\end{proof}

The second lemma, which will enable  us to complete the proof of Theorem \ref{th:cont} reads:
\begin{lem} \label{l:cont_2}
 Let $(\rho,u,\nabla\Pi)$ be the solution of System \eqref{eq:ddeuler} such
  that\footnote{with the usual convention that continuity in time is weak if $r=\infty.$}
\begin{itemize}
 \item $\rho\in\cC([0,T^*[:B^s_{\infty,r})$ and \eqref{eq:ldv};
 \item $u\in\cC([0,T^*[:B^s_{\infty,r})\cap\cC^1([0,T^*[;L^2)$;
 \item $\nabla\Pi\in\cC([0,T^*[;L^2)\cap L^1([0,T^*[;B^s_{\infty,r}).$
\end{itemize}

Moreover, suppose that 
$$\|u\|_{L^\infty_{T^*}(B^s_{\infty,r}\cap L^2)}+\|\nabla a\|_{L^\infty_{T^*}(B^{s-1}_{\infty,r})}<\infty.$$

Then $(\rho,u,\nabla\Pi)$ can be continued beyond time $T^*$ into a solution of \eqref{eq:ddeuler} with the same regularity.
\end{lem}

\begin{proof}[Proof of Lemma \ref{l:cont_2}]
From the proof of Theorem \ref{th:L^2} we know  that there exists a time $\veps$, depending only on 
$\rho^*$, $N,$ $s,$ 
$\|u\|_{L^\infty_{T^*}(B^s_{\infty,r}\cap L^2)},$ $\|\nabla a\|_{L^\infty_{T^*}(B^{s-1}_{\infty,r})}$
and on the norm of the data
such that, for all $\wtilde{T}<T$, Euler system with data $(\rho(\wtilde{T}),u(\wtilde{T}),f(\wtilde{T}+\cdot))$ has a unique solution
until time $\veps$.

Now, taking for example $\wtilde{T}=T-\veps/2$, we thus obtain a solution, which is the continuation of the initial one, $(\rho,u,\nabla\Pi)$,
until time $T+\veps/2$.
\end{proof}

Let us complete the proof of  Theorem \ref{th:cont}. The first part
 is a straightforward consequence of these two lemmas.
Indeed: Lemma \ref{l:cont_1} ensures that
 $\|u\|_{L^\infty_{T^*}(B^s_{\infty,r}\cap L^2)}$ and $\|\nabla a\|_{L^\infty_{T^*}(B^{s-1}_{\infty,r})}$ are
finite.
As for  the last claim (the Beale-Kato-Majda type continuation criterion),  it is a classical consequence of  the well-known logarithmic interpolation inequality (see e.g. \cite{BCD})
$$
\|\nabla u\|_{L^\infty}\leq C\left(\|u\|_{L^2}\,+\,
\|\Omega\|_{L^\infty}\log\left(e+\frac{\|\Omega\|_{B^{s-1}_{\infty,r}}}{\|\Omega\|_{L^\infty}}\right)\right)\,.
$$

So Theorem \ref{th:cont} is now completely proved, up to the endpoint case $s=r=1$.\qed


\section{The vorticity equation and applications}\label{s:W^1,4}

This section is devoted to the proof of the blow-up criterion in the endpoint case 
$s=r=1,$ and of Theorem \ref{th:W^1,4}.
Both results rely on the vorticity equation associated to   System \eqref{eq:ddeuler}.
 As done in Section \ref{s:L^2}, we shall  restrict ourselves to  the evolution  for positive times and 
 make the usual convention as regards time continuity, if $r<\infty$.

\subsection{On the vorticity}

As in all this section  the vorticity will  play a fundamental role, let us spend some words about it.
Given a vector-field $u$, we set $\nabla u$ its Jacobian matrix and $^t\nabla u$ the transposed
matrix of $\nabla u.$ We define the vorticity associated to $u$ by
$$
\Omega\,:=\,\nabla u\,-\,^t\nabla u.
$$

Recall that, in dimension $N=2$, $\Omega$ can be identified with the scalar function $\omega\,=\,\partial_1u^2\,-\,\partial_2u^1$, while
for $N=3$ with the vector-field $\omega\,=\,\nabla\times u$.

\medbreak
It is obvious that, for all $q\in[1,\infty]$, if $\nabla u\in L^q$, then also $\Omega\in L^q$.
Conversely, if  $u$ is divergence-free then for all $1\leq i\leq N$ we have $\Delta u^i\,=\,\sum_{j=1}^N\partial_j\Omega_{ij}$, and so, formally,
$$
\nabla u^i\,=\,-\nabla\left(-\Delta\right)^{-1}\sum_{j=1}^N\partial_j\,\Omega_{ij}\,.
$$
As the symbol of the operator $-\d_i\left(-\Delta\right)^{-1}\d_j$ is $\sigma(\xi)=\xi_i\xi_j/|\xi|^2,$
the classical Calderon-Zygmund Theorem ensures that\footnote{This time the extreme values are not included.}
 for all $q\in\,]1,\infty[$  if
$\Omega\in L^q$, then $\nabla u\in L^q$ and
\begin{equation}\label{eq:CZ}
\|\nabla u\|_{L^q}\leq C\|\Omega\|_{L^q}.
\end{equation}

The above relation also implies that 
$$
u=\Delta_{-1}u-(\Id-\Delta_{-1})(-\Delta)^{-1}\sum_j\d_j\Omega_{ij}.
$$
Hence combining Bernstein's inequality and Proposition \ref{p:CZ}, we gather that
\begin{equation}\label{eq:3}
\|u\|_{B^1_{\infty,1}}\leq C\bigl(\|u\|_{L^q}+\|\Omega\|_{B^0_{\infty,1}}\bigr)\quad\hbox{for all }\ 
q\in[1,\infty].
\end{equation}

{}From now on, let us assume that $\Omega$ is the vorticity
associated to some solution $(\rho,u,\nabla\Pi)$ of \eqref{eq:ddeuler},
defined on $[0,T]\times\R^N.$
{}From the velocity equation, we gather 
that $\Omega$ satisfies the following transport-like equation:
\begin{equation} \label{eq:vort}
 \d_t\Omega\,+\,u\cdot\nabla\Omega\,+\,\Omega\cdot\nabla u\,+\,^t\nabla u\cdot\Omega\,+\,\nabla\left(\frac{1}{\rho}\right)\wedge\nabla\Pi\,=\,F\,
\end{equation}
where $F_{ij}:=\d_jf^i-\d_if^j$ and, for two vector fields $v$ and $w$, we have set $v\wedge w$ to be the skew-symmetric matrix with components
$$
\left(v\wedge w\right)_{ij}\,=\,v^jw^i\,-\,v^iw^j\,.
$$

Using classical $L^q$ estimates for transport equations and taking advantage of Gronwall's Lemma, from (\ref{eq:vort}) we immediately get
\begin{eqnarray}
 \|\Omega(t)\|_{L^q} & \leq & \exp\left(\int^t_0\|\nabla u\|_{L^\infty}d\tau\right) \label{est:vort} \\
& & \times\left(\|\Omega(0)\|_{L^q}\,+\,\int^t_0e^{-\int^\tau_0\|\nabla u\|_{L^\infty}d\tau'}
\biggl(\|F\|_{L^q}+\left\|\frac{1}{\rho^2}\nabla\rho\wedge\nabla\Pi\right\|_{L^q}\biggr)\,d\tau\right). \nonumber
\end{eqnarray}

Let us notice that, in the case of space dimension $N=2$, equation (\ref{eq:vort}) becomes
$$
\partial_t\omega\,+\,u\cdot\nabla\omega\,+\,\nabla\left(\frac{1}{\rho}\right)\wedge\nabla\Pi\,=\,F\,,
$$
so that one obtains the same estimate as before, but without the exponential growth:
$$
\|\omega(t)\|_{L^q}\,\leq\,\|\omega(0)\|_{L^q}\,+\,\int^t_0
\biggl(\|F\|_{L^q}+\left\|\frac{1}{\rho^2}\nabla\rho\wedge\nabla\Pi\right\|_{L^q}\biggr)\,d\tau\,.
$$
Therefore, the two-dimensional case is in a certain sense better. We shall take
advantage of that in  Section \ref{s:lifespan}. 
As concerns the results of this section, the
proof will not depend on the dimension.  So for the time being we assume that 
the dimension $N$ is any integer greater than or equal to~$2.$


\subsection{Proof of Theorem \ref{th:cont} in the limit case $s=r=1$}

We just have to modify the proof of Lemma \ref{l:cont_1}. 
{}From the vorticity equation \eqref{eq:vort} and Proposition \ref{p:transport}
(recall that $\div u=0$), we readily get 
\begin{eqnarray}\label{eq:vort0}
&&\quad\|\Omega(t)\|_{B^0_{\infty,1}}\leq \exp\biggl(C\int_0^t\|\nabla u\|_{L^\infty}\,d\tau\biggr)
\\&&\hspace{1cm}\times
\biggl(\|\Omega_0\|_{B^0_{\infty,1}}+\int_0^t\|F\|_{B^0_{\infty,1}}\,d\tau
+\int_0^t\bigl(\|\nabla a\wedge\nabla\Pi\|_{B^0_{\infty,1}}+\|\Omega\cdot\nabla u+{}^t\nabla u\cdot\Omega\|_{B^0_{\infty,1}}\bigr)\,d\tau\biggr).\nonumber
\end{eqnarray}
We claim that 
\begin{eqnarray}\label{eq:vort1}
&&\|\nabla a\wedge\nabla\Pi\|_{B^0_{\infty,1}}\leq 
C\bigl(\|\nabla a\|_{L^\infty}\|\nabla\Pi\|_{B^0_{\infty,1}}+\|\nabla\Pi\|_{L^\infty}
\|a\|_{B^1_{\infty,1}}\bigr),\\[2ex]\label{eq:vort2}
&&
\|\Omega\cdot\nabla u+{}^t\nabla u\cdot\Omega\|_{B^0_{\infty,1}}
\leq C\|\nabla u\|_{L^\infty}\|u\|_{B^1_{\infty,1}}.
\end{eqnarray}
Both inequalities rely on Bony's decomposition  \eqref{eq:bony} and algebraic cancellations.
 Indeed, we observe that
$$
\displaylines{\d_ia\,\d_j\Pi-\d_ja\,\d_i\Pi=
T_{\d_ia}\d_j\Pi-T_{\d_ja}\d_i\Pi+T_{\d_j\Pi}\d_ia-T_{\d_i\Pi}\d_ja
+\d_iR(a,\d_j\Pi)-\d_jR(a,\d_i\Pi).}
$$
Applying Proposition \ref{p:op} thus yields \eqref{eq:vort1}.
\smallbreak
Next, we notice that, as $\div u=0,$
$$\begin{array}{lll}
\bigl(\Omega\cdot\nabla u+{}^t\nabla u\cdot\Omega\bigr)_{ij}&=&\Sum_k
\bigl(\d_iu^k\d_ku^j-\d_ju^k\d_ku^i\bigr),\\[2ex]
&=&\Sum_k\Bigl(\d_k\bigl(u^j\d_iu^k\bigr)-\d_k\bigl(u^i\d_ju^k\bigr)\Bigr).
\end{array}$$
Therefore, 
$$
\displaylines{\bigl(\Omega\cdot\nabla u+{}^t\nabla u\cdot\Omega\bigr)_{ij}\hfill\cr\hfill
=\sum_k\biggl(T_{\d_iu^k}\d_ku^j-T_{\d_ju^k}\d_ku^i+
T_{\d_ku^j}\d_iu^k-T_{\d_ku^i}\d_ju^k
+\d_kR(u^j,\d_iu^k)-\d_kR(u^i,\d_ju^k)\biggr).}
$$
Hence Proposition \ref{p:op} implies \eqref{eq:vort2}. 
\medbreak
It is now easy to complete the proof of Lemma \ref{l:cont_1} in the limit case. 
Indeed, plugging \eqref{eq:vort1} and \eqref{eq:vort2} in \eqref{eq:vort0}, using
the energy inequality \eqref{est:u_L^2} and Inequality \eqref{eq:3} with $q=2,$ we easily get
$$
\displaylines{
\|u(t)\|_{B^1_{\infty,1}\cap L^2}\leq C\exp\biggl(C\int\|\nabla u\|_{L^\infty}\,d\tau\biggr)
\hfill\cr\hfill\times\biggl(\|u_0\|_{B^1_{\infty,1}\cap L^2}
+\int_0^t\|f\|_{B^1_{\infty,1}\cap L^2}\,d\tau+
\int_0^t\bigl(\|\nabla a\|_{L^\infty}\|\nabla\Pi\|_{B^0_{\infty,1}}+\|\nabla\Pi\|_{L^\infty}
\|a\|_{B^1_{\infty,1}}\bigr)\,d\tau\biggr).}
$$

Hence, adding up Inequality \eqref{eq:q1} and using Gronwall's inequality,
we end up with 
$$
X(t)\leq C\exp\biggl(C\int_0^t\|(\nabla u,\nabla a,\nabla\Pi)\|_{L^\infty}\,d\tau\biggr)
\biggl(X(0)+\int_0^t\bigl(\|f\|_{B^1_{\infty,1}\cap L^2}+
\|\nabla a\|_{L^\infty}\|\nabla\Pi\|_{B^0_{\infty,1}}\bigr)\,d\tau\biggr)
$$
with $X(t):=\|a(t)\|_{B^1_{\infty,1}}+\|u(t)\|_{B^1_{\infty,1}\cap L^2}.$
\smallbreak
Taking advantage of \eqref{eq:Da}  completes the proof of Lemma \ref{l:cont_1} in the limit case. 
\qed


\subsection{Proof of Theorem \ref{th:W^1,4}}

We first prove a priori estimates, and then we will get from them existence and uniqueness of the solution. In fact, it will turn out to be possible to apply
 Theorem \ref{th:L^2} after  performing a suitable cut-off
 on the initial velocity field and thus  to work directly on System \eqref{eq:ddeuler}, without passing through the equivalence with (\ref{eq:ddeuler_a}) or with a sequence of approximate linear systems.

\subsubsection{A priori estimates}

 As in the previous section,
remembering also Remark \ref{r:a-rho}, the following estimates  hold true:
\begin{eqnarray}
 \|\nabla\rho(t)\|_{B^{s-1}_{\infty,r}} & \leq & \|\nabla\rho_0\|_{B^{s-1}_{\infty,r}}\,
\exp\left(C\int^t_0\|u\|_{B^{s}_{\infty,r}}\,d\tau\right) \label{est:drho_B} \\
 \|u(t)\|_{B^{s}_{\infty,r}} & \leq & \exp\biggl(C\int^t_0\|u\|_{B^{s}_{\infty,r}}\,d\tau\biggr)\,\cdot \,
\biggl(\|u_0\|_{B^s_{\infty,r}}\,+  \label{est:u_B} \\
& & +\,\int^t_0e^{-C\int^\tau_0\|u\|_{B^{s}_{\infty,r}}\,d\tau'}
\|\rho\|_{B^{s}_{\infty,r}}\|\nabla\Pi\|_{B^s_{\infty,r}}\,d\tau\biggr)\,. \nonumber
\end{eqnarray}
Moreover, from the transport equation satisfied by the velocity field, we easily gather that
$$ \|u(t)\|_{L^4}  \leq  \|u_0\|_{L^4}\,+\,\int^t_0\left\|\frac{\nabla\Pi}{\rho}\right\|_{L^4}\,d\tau.
$$
Therefore, using interpolation in Lebesgue spaces and embedding (see Corollary \ref{c:embed}),
\begin{eqnarray}
 \label{est:u-L^4}
 \|u(t)\|_{L^4} & \leq & \|u_0\|_{L^4}\,+\,\frac{1}{\rho_*}\int^t_0\|\nabla\Pi\|^{\frac{1}{2}}_{L^\infty}\,\|\nabla\Pi\|^{\frac{1}{2}}_{L^2}\,d\tau \nonumber \\
& \leq & \|u_0\|_{L^4}\,+\,\frac{C}{\rho_*}\int^t_0\|\nabla\Pi\|_{B^s_{\infty,r}\cap L^2}\,d\tau. 
\end{eqnarray}
In order to bound the vorticity in $L^4$, one may use that
\begin{eqnarray*}
 \left\|\frac{1}{\rho^2}\nabla\rho\wedge\nabla\Pi\right\|_{L^4} & \leq & \frac{1}{(\rho_*)^2}\left\|\nabla\rho\wedge\nabla\Pi\right\|_{L^4} \\
& \leq & \frac{1}{(\rho_*)^2}\,\|\nabla\rho\|_{L^\infty}\,\|\nabla\Pi\|_{L^4} \\
& \leq & \frac{C}{(\rho_*)^2}\,\|\nabla\rho\|_{B^{s-1}_{\infty,r}}\,\|\nabla\Pi\|_{B^s_{\infty,r}\cap L^2}\,.
\end{eqnarray*}
{}From this and (\ref{est:vort}), we thus get
\begin{eqnarray}
 \|\Omega(t)\|_{L^4} & \leq & \exp\left(\int^t_0\|\nabla u\|_{B^{s-1}_{\infty,r}}d\tau\right)\label{est:vort_2} \\
& & \times\left(\|\Omega_0\|_{L^4}\,+\,\frac{C}{(\rho_*)^2}\,\int^t_0e^{-\int^\tau_0\|\nabla u\|_{B^{s-1}_{\infty,r}}d\tau'}
\|\nabla\rho\|_{B^{s-1}_{\infty,r}}\,\|\nabla\Pi\|_{B^s_{\infty,r}\cap L^2}\,d\tau\right)\,. \nonumber
\end{eqnarray}

Now, in order to close the estimates, we need to control the pressure term. Its Besov norm can be bounded as in Section \ref{s:L^2}, up to a change
of $\|u\|_{L^2}$ into $\|u\|_{L^4};$  indeed it is clear that  in Inequality \eqref{eq:Pu}
the $L^2$ norm of $u$ may be replaced by any $L^q$ norm with $q<\infty.$
As a consequence, combining the (modified) inequality \eqref{eq:Fa}
and \eqref{est:Pi_Bes} yields
\begin{equation}\label{est:Pi_Bes4}
\|\nabla\Pi\|_{L^1_t(B^s_{\infty,r})}  \leq  C\left(\left(1+\|\nabla a\|^\gamma_{L^\infty_t(B^{s-1}_{\infty,r})}\right)
\|\nabla\Pi\|_{L^1_t(L^2)}\,+\,\|\rho\|_{L^\infty_t(B^{s}_{\infty,r})}
\int^t_0\|u\|^2_{B^{s}_{\infty,r}\cap L^4}d\tau\right).
\end{equation}

In order to bound the $L^2$ norm of $\nabla\Pi,$ we take
 the divergence of the second equation of System \eqref{eq:ddeuler}. We obtain
$$
-\div\left(\frac{\nabla\Pi}{\rho}\right)\,=\,\div\left(u\cdot\nabla u\right)\,,
$$
from which, applying elliptic estimates of Lemma \ref{l:laxmilgram} and 
\begin{equation}\label{est:du_L^4}
\|\nabla u\|_{L^4}  \leq  C\,\|\Omega\|_{L^4}\,, 
\end{equation}
 we get
\begin{equation} \label{est:Pi_L^2}
 \frac{1}{\rho^*}\,\|\nabla\Pi\|_{L^2}\,\leq\,\|u\cdot\nabla u\|_{L^2}\,\leq\,\|u\|_{L^4}\,\|\nabla u\|_{L^4}\,\leq\,C\|u\|_{L^4}\,\|\Omega\|_{L^4}\,.
\end{equation}

We claim that putting together inequalities \eqref{est:drho_B}, \eqref{est:u_B},
\eqref{est:u-L^4}, \eqref{est:du_L^4},  \eqref{est:Pi_Bes4}, \eqref{est:vort_2} and
\eqref{est:Pi_L^2}  enables us to close  the estimates on some nontrivial 
time interval $[0,T]$ depending only on the norm of the data. 

In effect, assuming that $T$ has been chosen so that Inequality \eqref{eq:smallu}
is satisfied, we get from the above inequalities
$$\begin{array}{lll}
\|u(t)\|_{B^s_{\infty,r}}&\leq& 2\|u_0\|_{B^s_{\infty,r}}
+C_0\|\nabla\Pi\|_{L_t^1(B^s_{\infty,r})},\\[2ex]
\|\nabla\Pi\|_{L^1_t(B^s_{\infty,r})}  &\leq&  C_0\biggl(\Int_0^t\bigl(\|u\|_{L^4}\|\Omega\|_{L^4}
+\|u\|_{B^s_{\infty,r}\cap L^4}^2\,d\tau\biggr),\\[2ex]
\|u(t)\|_{L^4}&\leq& \|u_0\|_{L^4}+C_0\|\nabla\Pi\|_{L_t^1(B^s_{\infty,r})}
+C_0\Int_0^t\|u\|_{L^4}\|\Omega\|_{L^4}\,d\tau,\\[2ex]
\|\Omega(t)\|_{L^4}&\leq& 2\|\Omega_0\|_{L^4}+C_0\|\nabla\Pi\|_{L^1_t(B^s_{\infty,r})},
\end{array}
$$
where the constant $C_0$ depends on $s,$ $\rho_*,$ $\rho^*,$ $N$ and 
$\|\rho_0\|_{B^1_{\infty,1}}.$ 
\smallbreak

Therefore, applying Gronwall lemma and assuming 
that $T$ has been chosen so that (in addition to \eqref{eq:smallu}) we have
$$
\int_0^T\|u\|_{W^{1,4}}\,d\tau\leq c
$$
where $c$ is a small enough constant depending only on $C_0,$ 
it is easy to close the estimates.

\begin{rem}\label{r:W^1,4}
Exhibiting an $L^2$ estimate for $\nabla\Pi$ \emph{even though $u$ is not in $L^2$}
is the key to the proof. 
This has been obtained in \eqref{est:Pi_L^2}. 
Note however that we have some freedom there.
In fact, one may rather assume that  $u_0\in L^p$ and $\nabla u_0\in L^q$, with $p$ and $q$ in $]2,\infty[$ such that $1/p\,+\,1/q\,\geq\,1/2$ and
get a statement  similar to that of  Theorem \ref{th:W^1,4}
under these two assumptions. The details are left to the reader.
\end{rem}


\subsubsection{Existence of a solution}

We want to take  advantage of the existence theory provided by  Theorem \ref{th:L^2}.
However, under the assumptions of Theorem \ref{th:W^1,4}, the initial 
velocity does not belong to $L^2.$ To overcome this,   we shall introduce 
a sequence of truncated initial velocities.  Then 
 Theorem \ref{th:L^2} will enable us to solve System \eqref{eq:ddeuler} with
 these modified data and the previous 
 part will provide uniform estimates
 in the right functional spaces on a small enough (fixed) time interval.
 Finally, convergence will be proved by an energy method similar
 to that we used for Theorem \ref{th:L^2}.

\subsubsection*{First step: construction of the sequence of approximate solutions}

Take any $\Phi\in C^\infty_0(\mbb{R}^N_x)$ with $\Phi\equiv1$ on a neighborhood of the origin, and set $\Phi_n(x)=\Phi(x/n).$
Then let us define $u^n_0\,:=\,\Phi_n\,u_0$ for all $n\in\mbb{N}.$

Given that $u^n_0$ is continuous and compactly supported, it obviously belongs to 
$L^2.$ Of course, we still have
$u^n_0\in B^s_{\infty,r}\cap W^{1,4}\cap L^2$, so we fall back into hypothesis of Theorem \ref{th:L^2}. {}From it, we get the existence of some time $T_n$ and of a solution $(\rho^n,u^n,\nabla\Pi^n)$ to \eqref{eq:ddeuler}  with data $(\rho_0,u_0^n,0)$ 
such that  $\rho^n\in\cC([0,T_n];B^s_{\infty,r})$, $u^n\in\cC^1([0,T_n];L^2)\cap\cC([0,T_n];B^s_{\infty,r})$ and $\nabla\Pi^n\in\cC([0,T_n];L^2)\cap L^1([0,T_n];B^s_{\infty,r})$. {}From \eqref{est:du_L^4}, the vorticity equation and the velocity equation, 
it is easy to see that, in addition,  $u^n\in\cC([0,T_n];W^{1,4})$.

Finally, as the norm of  $u_0^n$ in $W^{1,4}\cap B^s_{\infty,r}$
may be bounded independently of $n,$ the a priori estimates that have been performed
in the previous paragraph ensure that one may find some positive lower bound
$T$ for $T_n$ such that $(\rho^n,u^n,\nabla\Pi^n)$ satisfies bounds
independent of $n$ on $[0,T]$ in the desired functional spaces.

\subsubsection*{Second step: convergence of the sequence}

As done in the previous section, we define $\wtilde{\rho}^n\,=\,\rho^n-\rho_0$, and then
\begin{eqnarray*}
 \delta\! \rho^n & := & \wtilde{\rho}^{n+1}\,-\,\wtilde{\rho}^n\,, \\
\delta\! u^n & := & u^{n+1}\,-\,u^n\,, \\
\delta\!\Pi^n & := & \Pi^{n+1}\,-\,\Pi^n\,.
\end{eqnarray*}

Resorting to  the same type of computations as in the previous section (it is actually easier
as, now, $\div u^n=0$ for all $n$), we can prove that
$\left(\wtilde{\rho}^n,u^n,\nabla\Pi^n\right)_{n\in\mbb{N}}$ is a Cauchy sequence in $\cC([0,T];L^2).$
Hence it converges to some
$\left(\wtilde{\rho},u,\nabla\Pi\right)$ which belongs to the same space.



Now, defining  $\rho\,:=\,\rho_0\,+\,\wtilde{\rho}$, bearing in mind
the uniform estimates of the previous step, and using the Fatou property,
we easily conclude that
\begin{itemize}
 \item $\rho\in L^\infty([0,T];B^s_{\infty,r})$ and $\rho_*\leq\rho\leq\rho^*$;
 \item $u\in L^\infty([0,T];B^s_{\infty,r})\cap L^\infty([0,T]; W^{1,4})$;
 \item $\nabla\Pi\in L^1([0,T];B^s_{\infty,r})\cap L^\infty([0,T];L^2)$.
\end{itemize}
Finally, by interpolation we can pass to the limit in the equations at step $n$, so we get that $(\rho,u,\nabla\Pi)$ satisfies (\ref{eq:ddeuler}),
while continuity in time  follows from Proposition \ref{p:transport}.\qed


\section{Remarks on the lifespan of the solution} \label{s:lifespan}

In this section, we exhibit  lower bounds  for the lifespan
of the solution to the  density dependent incompressible Euler equations. 
We first establish that, like in the homogeneous case, in any dimension,
if the initial velocity is of order $\eps$  then 
the lifespan is at least of order $\eps^{-1}$
\emph{even if the density is not a small perturbation of a positive real number}.
Next we focus on the  two-dimensional case:
 we show in the second part of this section,
 that  for small perturbations of a constant density state, 
the lifespan tends to be very large.
Therefore, for nonhomogeneous  incompressible fluids too, the
two-dimensional case is somewhat nicer than
the general situation.

\subsection{The general case}

Let $\rho_0,$ $u_0$ and $f$ satisfy 
the assumptions of Theorem \ref{th:L^2} or \ref{th:W^1,4}.   Denote
$$
\tilde u_0(x):=\eps^{-1}u_0(x)\quad\hbox{and}\quad
\tilde f(t,x):=\eps^{-2}f(\eps^{-1}t,x).
$$
It is clear that if we set
$$
(\rho,u,\nabla\Pi)(t,x)=(\tilde\rho,\eps\tilde u,\eps^2\nabla\tilde\Pi)(\eps t,x)
$$
then $(\tilde\rho,\tilde u,\nabla\tilde\Pi)$
is a solution to \eqref{eq:ddeuler} on $[T_*,T^*]$ with 
data $(\rho_0,\tilde u_0,\tilde f)$ if and only if 
 $(\rho,u,\nabla \Pi)$
is a solution to \eqref{eq:ddeuler} on $[\eps^{-1}T_*,\eps^{-1}T^*]$ with 
data $(\rho_0,u_0,f).$

Hence, putting together the results of the previous section, 
we can  conclude to the following statement.
\begin{thm}\label{th:ND} Let $(\rho_0,\tilde u_0)$ satisfy the assumptions of Theorem \ref{th:L^2} or \ref{th:W^1,4}, and $f\equiv0.$
There exists a  positive time  $T^*$ depending only on $s,$ $N,$  $\rho_*,$
$\|\rho_0\|_{B^0_{\infty,1}}$ and $\|\tilde u_0\|_{B^0_{\infty,1}}$
such that  for any $\eps>0$ the upper bound $T^*_\eps$ of the maximal interval of existence
for the solution to \eqref{eq:ddeuler} with 
 initial data $(\rho_0,\eps\tilde u_0)$ satisfies
 $$
 T^*_\eps\geq \eps^{-1}T^*.
$$
A similar result holds for the lower bound of the maximal interval of existence.
\end{thm}


\subsection{The two-dimensional case}

Recall that for the homogeneous
equations,  any solution corresponding to suitably smooth data
is global, a fact which relies on the conservation of the vorticity
by the flow. Now, in our case, the vorticity equation reads (if $f\equiv0$)
\begin{equation}\label{eq:2D}
\d_t\omega+u\cdot\nabla\omega+\nabla b\wedge\nabla\Pi=0
\end{equation}
with $b:=1/\rho-1$ and $\nabla b\wedge\nabla\Pi:=\d_1 b\,\d_2\Pi-\d_2b\,\d_1\Pi.$
\medbreak
Owing to the new term involving the pressure and the nonhomogeneity, 
it is not clear at all that global existence still holds. Nevertheless,
we are going to prove that the lifespan  may be very large if the nonhomogeneity is small. 
\smallbreak
To simplify the presentation, we 
 focus on the case where
$\rho_0\in B^1_{\infty,1}(\R^2)$ and $u_0\in B^1_{\infty,1}(\R^2)$
(note that Corollary \ref{C:lifespan} ensures that this is not restrictive)
and  assume, in addition, that $u_0\in H^1(\R^2)$ 
(this lower order assumption may be somewhat relaxed too).

We aim at proving the following result.

\begin{thm}\label{th:2D}  Under the above assumptions
there exists a constant $c$ such that if
$b_0:=\frac1{\rho_0}-1$ satisfies
\begin{equation}\label{eq:smallb}
\|b_0\|_{B^1_{\infty,1}}\leq c
\end{equation}
then the lifespan of the solution to the two-dimensional density dependent
incompressible Euler equations with initial data $(\rho_0,u_0)$ and 
no source term  is bounded from below by
$$
\frac c{\|u_0\|_{H^1\cap B^1_{\infty,1}}}
\log\biggl(1+\log\frac c{\|b_0\|_{B^1_{\infty,1}}}\biggr)\cdotp$$
\end{thm}

\begin{proof}
Let $]T_*,T^*[$ denote the maximal interval of existence of the solution
$(\rho,u,\nabla\Pi)$  corresponding to $(\rho_0,u_0).$ 
To simplify the presentation, we focus on the evolution for \emph{positive} times.

The key to the proof relies on the fact that in the two-dimensional case, 
the vorticity equation satisfies \eqref{eq:2D}. 
Now, it turns out that, as discovered by M. Vishik in \cite{V}
and by T. Hmidi and S. Keraani in \cite{HK}, the 
norms in Besov spaces \emph{with null regularity index} of solutions
to transport equations satisfy better estimates, namely in our case
$$
\|\omega(t)\|_{B^0_{\infty,1}}
\leq C\biggl( \|\omega_0\|_{B^0_{\infty,1}}+\int_0^t\|\nabla b\wedge\nabla\Pi\|_{B^0_{\infty,1}}\,d\tau\biggr)
\biggl(1+\int_0^t\|\nabla u\|_{L^\infty}\,d\tau\biggr)
$$
whereas, according to Proposition \ref{p:transport},  the last term has to be replaced with
$\exp\Bigl(\Int_0^t\|\nabla u\|_{L^\infty}\,d\tau\Bigr)$ for nonzero regularity exponents. 
\smallbreak

Therefore, using Inequality \eqref{eq:vort1}, we get
\begin{equation}\label{eq:1}
\|\omega(t)\|_{B^0_{\infty,1}}
\leq C\biggl( \|\omega_0\|_{B^0_{\infty,1}}+\int_0^t\|b\|_{B^1_{\infty,1}}\|\nabla\Pi\|_{B^0_{\infty,1}}\,d\tau\biggr)
\biggl(1+\int_0^t\|\nabla u\|_{L^\infty}\,d\tau\biggr)
\end{equation}

Of course, a basic energy argument leads to 
\begin{equation}\label{eq:2}
\|\omega(t)\|_{L^2}\leq\|\omega_0\|_{L^2}+\int_0^t\|\nabla b\|_{L^\infty}\|\nabla\Pi\|_{L^2}\,d\tau
\end{equation}
and it is well-known that for two-dimensional divergence-free vector fields, we have
$$
\|\nabla u\|_{L^2}=\|\omega\|_{L^2}.
$$
Therefore  putting together Inequalities \eqref{eq:1} and \eqref{eq:2} and bearing 
in mind Inequality \eqref{eq:3} and the energy inequality for $u,$ we get
\begin{equation}\label{eq:4}
X(t)\leq C\biggl(X_0+\int_0^t B\,\|\nabla\Pi\|_{B^0_{\infty,1}\cap L^2}\,d\tau\biggr)
\biggl(1+\int_0^tX\,d\tau\biggr)
\end{equation}
with 
$$
X(t):=\|u(t)\|_{H^1\cap B^1_{\infty,1}}\quad\hbox{and}\quad
B(t):=\|b(t)\|_{B^1_{\infty,1}}.
$$
Bounding $B$ is easy given  that 
$$
\d_tb+u\cdot\nabla b=0.
$$
Indeed, Inequality \eqref{sanspertes1}
ensures that
$$
\|b(t)\|_{B^1_{\infty,1}}\leq \|b_0\|_{B^1_{\infty,1}}                
\exp\biggl(C\int_0^t\|\nabla u\|_{B^0_{\infty,1}}\,d\tau\biggr).
$$
Therefore, 
\begin{equation}\label{eq:5}
B(t) \leq B_0             
\exp\biggl(C\int_0^tX\,d\tau\biggr).  
\end{equation}

Bounding the pressure term in $B^0_{\infty,1}\cap L^2$ is our next task. 
For that, recall that, as
$$
\div(a\nabla\Pi)=-\div(u\cdot\nabla u),
$$
Lemma \ref{l:laxmilgram} guarantees that
\begin{equation}\label{eq:6}
a_*\|\nabla\Pi\|_{L^2}\leq \|u\|_{L^2}\|\nabla u\|_{L^\infty}.
\end{equation}
Next, differentiating once the pressure  equation and applying again an energy method yields
\begin{equation}\label{eq:7}
a_*\|\nabla^2\Pi\|_{L^2}\leq \|\nabla u\|_{L^2}\|\nabla u\|_{L^\infty}
+\|\nabla a\|_{L^\infty}\|\nabla\Pi\|_{L^2}.
\end{equation}
Therefore, combining \eqref{eq:6} and \eqref{eq:7}  and using elementary embedding, we get
\begin{equation}\label{eq:8}
\|\nabla\Pi\|_{H^1}\,\leq\,C\,\|a\|_{B^1_{\infty,1}}\,X^2.
\end{equation}
Note that $\|a\|_{B^1_{\infty,1}}$ and $A:=1+B$  are of the same order. 
This will be important in the sequel. 
\smallbreak
In order to bound the pressure term in $B^0_{\infty,1}$ we shall use 
the following classical logarithmic interpolation inequality (see e.g. \cite{BCD}, Chap. 2):
\begin{equation}\label{eq:9}
\|\nabla\Pi\|_{B^0_{\infty,1}}\leq C\|\nabla\Pi\|_{H^1}\log\biggl(e+\frac{\|\nabla\Pi\|_{B^1_{\infty,1}}}{\|\nabla\Pi\|_{H^1}}\biggr).
\end{equation}

In order to estimate $\|\nabla\Pi\|_{B^1_{\infty,1}},$ we use the identity 
$$
\nabla\Pi=\Delta_{-1}\nabla\Pi+\cA(D)\div(u\cdot\nabla u)+\cA(D)\div(b\nabla\Pi).
$$
 with $\cA(D):=(-\Delta)^{-1}\nabla({\rm Id}-\Delta_{-1}).$
 \smallbreak
 On the one hand, combining Bony's decomposition 
 with the fact that  $\div(u\cdot\nabla u)=\nabla u:\nabla u,$
it is easy to show that 
$$
\|\div(u\cdot\nabla u)\|_{B^0_{\infty,1}}\leq 
\|u\|_{B^1_{\infty,1}}^2.
$$
On the other hand, Proposition \ref{p:op} guarantees that
$$
\|b\nabla\Pi\|_{B^1_{\infty,1}}\leq C\Bigl(\|b\|_{L^\infty}\|\nabla\Pi\|_{B^1_{\infty,1}}
+\|\nabla\Pi\|_{L^\infty}\|b\|_{B^1_{\infty,1}}\Bigr).
$$

So using the fact that $\cA(D)$ (resp. $\cA(D)\div$)  is a multiplier of degree $-1$ (resp. $0$)
away from the origin, we get from Proposition \ref{p:CZ}, 
$$
\|\nabla\Pi\|_{B^1_{\infty,1}}\leq C\Bigl(\|\nabla\Pi\|_{L^2}+ \|u\|_{B^1_{\infty,1}}^2+
\|\nabla\Pi\|_{L^\infty}\|b\|_{B^1_{\infty,1}}+\|b\|_{L^\infty}\|\nabla\Pi\|_{B^1_{\infty,1}}\Bigr).
$$
Note that $\|b(t)\|_{L^\infty}$ is time independent and that $B^1_{\infty,1}\hookrightarrow L^\infty.$ 
Hence, under assumption \eqref{eq:smallb} with $c$ small enough, the last
term may be absorbed by the left-hand side. As regards the last but one term, 
we use the following interpolation inequality:
$$
\|\nabla\Pi\|_{L^\infty}\leq C\|\nabla\Pi\|_{L^2}^{\frac12}\|\nabla\Pi\|_{B^1_{\infty,1}}^{\frac12}.
$$
Combining with Young's inequality, we thus conclude that, 
under assumption \eqref{eq:smallb}, we have
$$
\|\nabla\Pi\|_{B^1_{\infty,1}}\leq C\Bigl(\|u\|_{B^1_{\infty,1}}^2+(1+\|b\|_{B^1_{\infty,1}}^2)
\|\nabla\Pi\|_{L^2}\Bigr).
$$
Bounding the last term according to  \eqref{eq:6}, we thus end up with
$$
\|\nabla\Pi\|_{B^1_{\infty,1}}\leq CA^2X^2.
$$

Inserting this inequality in \eqref{eq:9} and using also \eqref{eq:8}, one may now conclude
that
\begin{equation}\label{eq:10}
\|\nabla\Pi\|_{B^0_{\infty,1}\cap L^2}\leq 
CAX^2\log(e+B).
\end{equation}

It is now time to insert Inequalities \eqref{eq:5} and \eqref{eq:10} in \eqref{eq:4}; we get
\begin{equation}\label{eq:11}
X(t)\leq C\biggl(X_0+B_0A_0\log(e+B_0)
\int_0^t e^{C\int_0^\tau X\,d\tau'}X^2\,d\tau\biggr)
\biggl(1+\int_0^tX\,d\tau\biggr).
\end{equation}
Let $T_0$ denote the supremum of times $t\in[0,T^*[$ so that 
\begin{equation}\label{eq:12}
B_0A_0\log(e+B_0)\int_0^t e^{C\int_0^\tau X\,d\tau'}X^2\,d\tau\leq X_0.
\end{equation}
 {}From \eqref{eq:11} and Gronwall's Lemma, we gather that 
 $$
 X(t)\leq 2CX_0e^{2CtX_0}
 \quad\hbox{for all}\quad
 t\in[0,T_0[.
 $$
 Note that this inequality implies that for all $t\in[0,T_0[,$ we have
 $$
 \int_0^t e^{C\int_0^\tau X\,d\tau'}X^2\,d\tau\leq 
 CX_0\biggl(e^{4CtX_0}-1\biggr)\exp\biggl(C\biggl(e^{2CtX_0}-1\biggr)\biggr).  
 $$
 Therefore, using \eqref{eq:12} and a bootstrap argument (based on the continuation 
 theorems that we proved in the previous sections), it is easy to show that $T_0$ is greater than 
 any time $t$ such that 
 $$
A_0B_0\log(e+B_0) \biggl(e^{4CtX_0}-1\biggr)\exp\biggl(C\biggl(e^{2CtX_0}-1\biggr)\biggr)\leq 1.         
 $$
 Taking the logarithm and using that $\log y\leq y-1$ for $y>0,$ we see that 
 if $B_0$ is small enough (an assumption which implies in particular that
$A_0\log(e+B_0)$ is of order $1$) 
  then  the above inequality is satisfied
 whenever
 $$
 e^{2CtX_0}-1\leq\frac1{C+2}\log\biggl(\frac1{2CB_0}\biggr).
 $$
 This completes the proof of the lower bound for $T^*.$
 \end{proof}

\begin{rem} If $\omega_0$ has more regularity (say $\omega_0\in C^r$ for some $r\in(0,1)$)
then one may first write an estimate for $\|\omega\|_{L^\infty}$
and next use the classical logarithmic inequality for bounding $\|\nabla u\|_{L^\infty}$
in terms of $\|\omega\|_{L^\infty}$ and $\|\omega\|_{C^r}.$
The proof is longer, requires more regularity and, at the same time, the lower
bound for the lifespan does not improve.
\end{rem}

\end{document}